\documentclass[smallextended]{svjour3}       
\smartqed  
\usepackage{mathrsfs}
\usepackage{amssymb,amsmath,amsfonts}
\usepackage{graphicx,subfigure,epstopdf}
\usepackage[usenames]{color}
\usepackage{url}
\usepackage{algorithm,algorithmic}
\usepackage{dsfont,verbatim}
\usepackage[colorlinks,linktocpage,linkcolor=blue]{hyperref}
\usepackage{enumerate,enumitem}
\usepackage[normalem]{ulem}
\usepackage{cite}
\usepackage{bm}
\allowdisplaybreaks

\usepackage[T1]{fontenc}
\usepackage[utf8]{inputenc}
\usepackage[english]{babel} 

\numberwithin{theorem}{section}
\numberwithin{lemma}{section}
\numberwithin{remark}{section}
\numberwithin{equation}{section}

\def\d{{\rm d}}

\def\R{\mathbb{R}}
\def\C{\mathbb{C}}
\def\D{\mathbb{D}}

\def\i{\mathrm{i}}

\def\ztr{\frac{e^{-\tau z} + e^{\tau z} - 2}{z^2 \tau}}
\def\ztrt{\frac{e^{-\tau z} + e^{\tau z} - 2}{z^2 \tau^2}}

\newcommand{\Ppol}[1]{\mathcal{P}_{#1}}

\newcommand{\vertiii}[1]{{\left\vert\kern-0.25ex\left\vert\kern-0.25ex\left\vert #1 
    \right\vert\kern-0.25ex\right\vert\kern-0.25ex\right\vert}}

\definecolor{green}{rgb}{0,0.7,0}

\definecolor{darkred}{rgb}{.7,0,0}

\begin{document}

\title{Weak maximum principle of finite element methods\\ for parabolic equations in polygonal domains}
\titlerunning{\,}        

\author{Genming Bai \and Dmitriy Leykekhman \and Buyang Li 
}

\authorrunning{\,} 

\institute{G. Bai and B. Li \at
              Department of Applied Mathematics, The Hong Kong Polytechnic University, Hung Hom, Hong Kong. 
              Email address: genming.bai@connect.polyu.hk, buyang.li@polyu.edu.hk\\
           \and
              D. Leykekhman \at
              Department of Mathematics, University of Connecticut, Storrs, CT 06269, USA. 
              Email address: dmitriy.leykekhman@uconn.edu
 }

\date{Received: date / Accepted: date}

\maketitle

\begin{abstract}
The weak maximum principle of finite element methods for parabolic equations is proved for both semi-discretization in space and fully discrete methods with $k$-step backward differentiation formulae for $k = 1,\dots,6$, on a two-dimensional general polygonal domain or a three-dimensional convex polyhedral domain. The semi-discrete result is established via a dyadic decomposition argument and local energy estimates in which the nonsmoothness of the domain can be handled. The fully discrete result for multistep backward differentiation formulae is proved by utilizing the solution representation via the discrete Laplace transform and the resolvent estimates, which are inspired by the analysis of convolutional quadrature for parabolic and fractional-order partial differential equations. \\
\end{abstract}

\setlength\abovedisplayskip{3.5pt}
\setlength\belowdisplayskip{3.5pt}

\section{Introduction}\label{section:introduction}
Maximum principle serves as a fundamental mathematical tool for the study of elliptic and parabolic partial differential equations (PDEs). The discrete counterpart of maximum principle associated to finite element methods (FEMs) has a well-established history of research and remains an active field. However, unlike its continuous counterpart, the discrete maximum principle is not an inherent property and is significantly influenced by the triangulation of the physical domain \cite[\textsection 5]{WangZhang_2012}. 

In two dimensions, the applicability of the discrete maximum principle for elliptic equations is primarily confined to piecewise linear elements, as quadratic elements necessitate uniform equilateral triangulation \cite{Hohn_Mittelmann_1981}. The scenario becomes more intricate in three dimensions \cite{Brandts_Korotov_Krizek_2009, Korotov_Krizek_2001, Korotov_Krizek_Pekka_2001, Xu_Zikatanov_1999}. Particularly, it becomes challenging to assure the discrete maximum principle for even piecewise linear elements. Nonetheless, a large number of applications do not require a strong discrete maximum principle. Schatz demonstrated in \cite{Schatz80} that a weak maximum principle (also known as the Agmon-Miranda principle) is applicable to a broad spectrum of finite elements on general quasi-uniform triangulation in any two-dimensional polygonal domains. This weak maximum principle has recently been extended to three-dimensional convex polyhedral domains \cite{Leykekhman_Li_2021} and the Neumann problem in two-dimensional polygonal domains \cite{Leykekhman_Li_2017, LiBuyang_2022}. The situation for parabolic equations is more complex, depending not only on space discretization but also on time discretization. 

This paper focuses on the weak maximum principle of semi-discrete and fully discrete FEMs for the following parabolic equation: 
\begin{align}\label{eq: PDE}
	\left\{
	\begin{aligned}
		\partial_t u - \Delta u &= 0 &&\mbox{in}\,\,\, (0,\infty) \times \Omega,\\ 
		u&=g &&\mbox{on}\,\,\,  [0,\infty) \times \partial\Omega ,\\
		u|_{t=0}&=u(0) &&\mbox{in}\,\,\,\Omega .
	\end{aligned}
	\right.	
\end{align} 
Specifically, we are interested in determining whether the semi-discrete and fully discrete finite element solutions of \eqref{eq: PDE} are uniformly bounded by constant multiples of $\|g\|_{L^\infty(\partial\Omega)}$ and $\|u(0)\|_{L^\infty(\Omega)}$. 

The discrete maximum principle of fully discrete FEMs for parabolic equations was initially addressed by Fujii in \cite{Fujii_1973}, demonstrating that the backward Euler method may satisfy a maximum principle for specific families of acute triangulations. However, his results require a lower bound for the time step size, and therefore do not imply the discrete maximum principle for the semi-discrete FEM (which would necessitate the step size tending to zero). In \cite{Thomee_Wahlbin_2008}, it was established that without the application of the mass lumping method, the strong discrete maximum principle for the semi-discrete FEM does not hold, even for acute triangulation, thereby making the mass lumping method a necessary requirement. These results were extended to more general conditions as well as more general single-step time-stepping methods (again, with a lower bound for the step size) and numerically illustrated in \cite{Chatzipantelidis_Horvath_Thomee_2015}. To the best of our knowledge, the weak discrete maximum principle of the semi-discrete FEM for parabolic equations has not been addressed, a gap this paper aims to fill.

The weak maximum principle of the semi-discrete FEM with respect to the initial value is equivalent to the uniform boundedness of the finite element heat semigroup in $L^\infty(\Omega)$. The latter was proved in \cite{Schatz98} for smooth domains by assuming that the partition of the domain contains curved simplices which fit the geometry of the boundary exactly. The extension to nonsmooth polygonal and polyhedral domains, thus a standard triangulation with flat simplices can fit the boundary exactly, introduces more technical challenges where the elliptic regularity theory and finite element error estimates degenerate. A comprehensive solution to these issues, which emerge due to the domain's nonsmoothness, was presented in a recent paper \cite{Li19}. However, the weak maximum principle of the semi-discrete FEM with respect to the boundary value $g$ remains unaddressed, and this issue can't be transformed into the uniform boundedness of the semigroup in $L^\infty(\Omega)$. This is proved in the current paper by the representation formula via discrete Laplace transform, discrete resolvent estimates on the contour and the consistency of discrete and continuous Laplace transform together with the elliptic weak maximum principle developed in \cite{Leykekhman_Li_2021}.

Regarding time discretizations, the maximal $L^pL^q$ regularity for multistep BDF and Runge-Kutta time-stepping methods, with $1<p,q<\infty$, was shown in \cite{KLL16} and \cite{Li22} for semi-discretization in time and fully discrete FEMs, respectively. The authors utilized the $R$-boundedness concept to reduce maximal $L^pL^q$ regularity to set inclusion \cite[Theorem 1.11]{KW04}, which is further characterized by the concept of $A(\alpha)$-stability. In addition to BDF, the discontinuous Galerkin time-stepping method was studied in \cite{Ley17}, and the $\theta$-scheme was considered in \cite{Kem18}. These findings are closely tied to the weak maximum principle of time discretizations with respect to the initial value. However, the weak maximum principle of time discretizations with respect to the boundary value has not been addressed. This question is answered affirmatively in the current paper for $k$-step BDF methods with $k = 1,\dots,6$. The analysis was particularly challenging due to the existence of boundary data $g$, the domain's nonsmoothness, and the lack of $A$-stability for $k$-step BDF with $3\leq k\leq 6$. The two main points of analysis in this paper are: 
\begin{enumerate}
	\item[\footnotesize$\bullet$]
	The weak maximum principle of the semi-discrete FEM requires proving the following estimate: 
	\begin{align}\label{eq:F_err}
		\| \Gamma_h - \Gamma \|_{L^1(0,\infty; L^1(D_h))} \leq h ,
	\end{align}
	where $D_h=\{x\in\Omega:{\rm dist}(x,\partial\Omega)\le h\}$, $\Gamma$ and $\Gamma_h$ are the parabolic Green's functions with respect to the regularized delta function $\tilde\delta$ (see \cite[Lemma 2.2]{Tho00}) and the discrete delta function $\delta_h $ (to be defined in Section \ref{sec:Green}), respectively. This error estimate is not straightforward due to the lack of an approximation property for the initial value $\Gamma_h(0) - \Gamma(0) = \delta_h - \tilde\delta = (P_h- 1) \tilde\delta$ in the error equation (where $P_h$ denotes the $L^2$ orthogonal projection onto the finite element space), as the initial value $\tilde \delta$ is only bounded by a universal constant $C$ in the (very weak) $L^1$ norm. The standard nonsmooth data error estimate \cite[Chapter 3]{Tho06} and the negative norm error estimate \cite[Chapter 5]{Tho06} are both hindered in the case of a nonsmooth domain. Note that the error estimate \eqref{eq:F_err} is applicable when the underlying domain is sufficiently smooth \cite[Proposition 3.2]{Schatz98}. However, for a nonconvex polygon, the argument of \cite[Proposition 3.2]{Schatz98} is no longer valid due to the degeneracy of the consistency error and the finite element error estimates. The solution is to leverage the local H\"older's inequality, the smoothing property of the discrete heat semigroup, and the lower bound $\alpha > \frac12$, where $\alpha$ is the indicator of the domain's smoothness \cite[Lemma 4.2]{Li19}. The lower bound $\frac12$ is sharp in our proof in the sense that our proof does not work if $\alpha < \frac12 $.

	\medskip\vspace{-3pt}
	
	\item[\footnotesize$\bullet$]
	Weak maximum principle of time discretizations is accomplished by studying the representation formula of $u_h^n$ in terms of the generating functions $\tilde u(\zeta)=\sum_{n=k}^\infty u_h^n$, which may alternatively be understood as the discrete Laplace transform or $z$-transform of the fully discrete finite element solution. For the representation of $u_{h,2}^n$ — the second part of the fully discrete solution — readers are referred to equations \eqref{eq:gen_iden}, \eqref{hat-uh-z}, and \eqref{u_{h,2}-tn-repr}. 

However, it is crucial to note that a direct application of the $L^\infty$ norm to the representation formula of $u_h^n$ introduces a logarithmic factor. This factor emerges from the summation process, which is a consequence of the discrete convolution as in convolutional quadrature \cite{Lubich88}, \cite[Chapter 3]{JZ23}:
\begin{align}\label{summ-1}
 \sum_{j=1}^{\lfloor T/\tau \rfloor} \frac{1}{j} \approx \log \frac{T}{\tau} .
\end{align}
In order to eliminate this logarithmic factor, a comparison is conducted between the representation formulas of the discrete and continuous Laplace transforms, as outlined in equation \eqref{BDF-FEM-difference}. This comparison necessitates the development of consistency between these two types of Laplace transform.

We draw attention to a central correspondence relation concerning the boundary data, as stated in Lemma \ref{lemma:z-trans}:
\begin{align}
 \hat g_h (z) = \frac{e^{-\tau z} + e^{\tau z} - 2}{z^2 \tau} \tilde g_h(e^{-\tau z}), \notag
\end{align}
Here, $\tilde g_h(\zeta)$ signifies the generating function (the discrete Laplace transform or $z$-transform) of $g_h^n$, while $\hat g_h (z)$ represents the continuous Laplace transform of $g_h (t)$. The latter transform corresponds to the piecewise linear reconstruction of the nodal values $g_h^n$. The removal of the logarithmic factor is exactly accomplished by utilizing this consistency, which improves \eqref{summ-1} to the following result: 
\begin{align}
 \sum_{j=1}^{\infty} \frac{1}{j^2} \leq C.
\end{align}
Furthermore, it should be noted that BDF-$k$ is not $A$-stable for $3\leq k \leq 6$. This instability necessitates the adherence to the BDF-$1$ symbol within the resolvent operator, i.e., $(\frac{1 - e^{-\tau z}}{\tau} - \Delta_h )^{-1}$. Consequently, to eliminate the logarithmic factor for BDF-$k$, consistency of the solution maps associated with BDF-$1$ and BDF-$k$ is required, as outlined in Lemma \ref{lemma:M-L}.
\end{enumerate}

The rest of this article is organized as follow. In Section \ref{sec: main results} we state the two main theorems regarding the weak maximum principle of semi-discrete and fully discrete FEMs for parabolic equations. In Section \ref{sec: semi-discrete}, we introduce the notation and establish the weak discrete maximum principle of the semi-discrete FEM. In Section \ref{sec: fully discrete}, we
prove the fully discrete weak maximum principle of fully discrete FEMs with multistep BDF time-stepping methods. Finally, in the conclusion section, we briefly discuss the extension to discontinuous Galerkin time-stepping methods with possibly highly variable time step sizes, as well as the limitation of current approach and the possibility to address it based on novel approaches developed recently.  

\section{The main results}\label{sec: main results}
\setcounter{equation}{0}

Let $S_h$, $0<h<h_0$, be a family of Lagrange finite element subspaces of $H^1(\Omega)$ consisting of all piecewise polynomials of degree $r\ge 1$ subject to a quasi-uniform triangulation of a polygonal/polyhedral domain $\Omega$. We denote by $\mathring S_h$ the subspace of $S_h$ with zero boundary conditions, and denote by $I_h:C(\overline\Omega)\rightarrow S_h$ the Lagrangian interpolation operator onto $S_h$. 

The semi-discrete FEM (FEM) for \eqref{eq: PDE} is to find $u_h\in C^1([0,\infty);S_h)$ satisfying the following weak formulation: 
\begin{align}\label{semi-discrete-FEM}
\left\{
\begin{aligned}
&(\partial_t u_h(t),v_h) + (\nabla u_h(t),\nabla v_h)=0 &&\forall\, v_h\in \mathring S_h, &&\forall\,t\in(0,\infty), \\ 
&u_h(t)=g_h(t) &&\mbox{on}\,\,\,\partial\Omega , &&\forall\,t\in [0,\infty)\\
&u_h(0)= u_h^0 &&\mbox{in}\,\,\,\Omega ,
\end{aligned}
\right.	
\end{align}
where $u_h^0 = I_h u(0) \in S_h(\Omega)$ and $g_h\in S_h(\partial\Omega)$ is the Lagrangian interpolation of $g$ on the boundary $\partial\Omega$. If $\tilde g$ is any extension of $g$ from $C(\partial\Omega)$ to $C(\overline\Omega)$, then $g_h=I_h\tilde g$ on $\partial\Omega$. 

The first main result of this article is the following theorem.

\begin{theorem}[Weak maximum principle of semi-discrete FEM]\label{thm:semi-WMP}
{\it
If $\Omega$ is a polygon (possibly nonconvex) in $\R^2$ or a convex polyhedron in $\R^3$, then the solution of the semi-discrete FEM in \eqref{semi-discrete-FEM} satisfies the following weak maximum principle for $r\ge 1$ in $\R^2$ and $r\ge 2$ in $\R^3$: 
\begin{align}
\|u_h\|_{L^\infty(0,T;L^\infty(\Omega))}
\le C\|u_h^0\|_{L^\infty(\Omega)} +
C\|g_h\|_{L^\infty(0,T;L^\infty(\partial\Omega))} , \notag
\end{align}
where $C$ is a constant independent of $h$ and $T$.
}
\end{theorem}

We denote by $\delta_j, j=0,...,k$, the coefficients of the generating polynomial of the BDF-$k$ method (see \cite[equation (3.5)]{JZ23}), i.e., 
\begin{align*}
	\delta(\zeta) = \sum_{j=1}^k \frac{1}{j} (1 - \zeta)^j = \sum_{j=0}^k\delta_j \zeta^j ,
\end{align*}
and use the notation $\delta_\tau(\zeta) = \tau^{-1} \delta(\zeta)$, where $\tau$ denotes the step size of time discretization. The BDF-$k$ method is know to be $A(\theta_k)$-stable, i.e., the set $\{z \in \C \backslash \{0\} : | \arg(z)| < \theta_k \}$ is contained in the stability region of the method (see \cite[p. 251]{WH96}), where $\theta_k = 90^\circ$, $90^\circ$, $86.03^\circ$, $73.35^\circ$, $51.84^\circ$, $17.84^\circ$ for $k = 1, 2, 3, 4, 5, 6$, respectively.

For the BDF-$k$ method, we assume that the starting values $u_h^n$, $n=0,...,k-1,$ are given or computed from some other time-stepping methods (such as the Taylor expansion method or the Runge--Kutta method).

Given any final time $T \geq (k-1)\tau$, we set $N = \lfloor T/\tau \rfloor $. For $n=k,..., N$, the fully discrete FEM with BDF-$k$ time-stepping method seeks $u_h^n\in S_h$ such that 
\begin{align}\label{Euler-FEM}
	\left\{
	\begin{aligned}
		&\bigg( \frac{1}{\tau} \sum_{j=0}^k \delta_j u_h^{n-j} ,v_h\bigg) + (\nabla u_h^n,\nabla v_h)=0 &&\forall\, v_h\in \mathring S_h, &&n=k,\dots, N , \\[2pt]
		&u_h^n=g_h^n &&\mbox{on}\,\,\,\partial\Omega ,
		&&n=k,\dots, N , 
		\end{aligned}
	\right.	
\end{align}
where $g_h^n = I_h g(t_n) \in S_h(\partial\Omega)$.

The second main result of this article is the following theorem.

\begin{theorem}[Weak maximum principle of fully discrete FEM]\label{THM:BDF-FEM}
	{\it If $\Omega$ is a polygon (possibly nonconvex) in $\R^2$ or a convex polyhedron in $\R^3$, then the fully discrete solution given by \eqref{Euler-FEM} satisfies the following weak maximum principle for $r\ge 1$ in $\R^2$ and $r\ge 2$ in $\R^3$: 
		\begin{align}
			\max_{k\le n\le N}\|u_h^n\|_{L^\infty(\Omega)}
			\le 
			C \max_{0\le n\le k-1} \|u_h^n\|_{L^\infty(\Omega)}
			+ 
			C\max_{k\le n\le N}\|g_h^n\|_{L^\infty(\partial\Omega)}, \notag
		\end{align}
		where $C$ is a constant independent of $\tau$, $h$ and $N$.
	}
\end{theorem}

\begin{remark}
From the proof of Theorems \ref{thm:semi-WMP} and \ref{THM:BDF-FEM} in the following sections we can see that, for a convex polyhedron in $\R^3$ and finite elements of degree $r=1$, the semi-discrete and fully discrete FEMs  satisfies the following weak maximum principle with an additional logarithmic factor, i.e., 
\begin{align}
\|u_h\|_{L^\infty(0,T;L^\infty(\Omega))}
&\le 
C\ln(2+1/h) \|g_h\|_{L^\infty(0,T;L^\infty(\partial\Omega))} , \\
\max_{k\le n\le N}\|u_h^n\|_{L^\infty(\Omega)}
			&\le 
			C \max_{0\le n\le k-1} \|u_h^n\|_{L^\infty(\Omega)}
			+ 
			C\ln(2+1/h) \max_{k\le n\le N}\|g_h^n\|_{L^\infty(\partial\Omega)} .
\end{align}
\end{remark}

\section{Notation, finite element discretization, and regularized Green's function}\label{Sec:notation}

\subsection{Function spaces and notation}

We use the conventional notations of Sobolev spaces $W^{s,q}(\Omega)$, $s\ge 0$ and $1\leq q\leq\infty$, 
with abbreviations $L^q=W^{0,q}(\Omega)$, $W^{s,q}=W^{s,q}(\Omega)$ and $H^s:=W^{s,2}(\Omega)$. 
The notation $H^{-s}(\Omega)$ denotes the dual space of $H^s_0(\Omega)$. The latter is defined as the closure of $C^\infty_0(\Omega)$ in $H^s(\Omega)$. 

For any given $W^{s,q}$-valued function $f:(0,T)\rightarrow W^{s,q}$, we can define the following Bochner norm: 
\begin{align}
	&\|f\|_{L^p(0,T;W^{s,q})} = 
	\big\| \|f(\cdot)\|_{W^{s,q}}\big\|_{L^p(0,T)} ,\quad\forall\,\, 1\leq p,q\leq \infty,\,\, s\in\R   . 
\end{align} 
For any subdomain $D\subset \Omega$, we define 
\begin{align}\label{Def-HsD}
	\|f\|_{W^{s,q}(D)}:=
	\inf_{\tilde f|_D=f}\|\tilde f\|_{W^{s,q}(\Omega)} 
	,\quad\forall\,\, 1\leq q\leq \infty ,\,\, s\in\R ,
\end{align} 
where the infimum extends over all possible
$\tilde f$ defined on $\Omega$ such that
$\tilde f=f$ in $D$. 
Similarly, for any subdomain $Q\subset {\mathcal Q}=(0,1)\times\Omega$, 
we define 
\begin{align}\label{DefLpX}
	\|f\|_{L^pW^{s,q}(Q)}:=
	\inf_{\tilde f|_Q=f}\|\tilde f\|_{L^p(0,T;W^{s,q})} 
	,\quad\forall\,\, 1\leq p,q\leq \infty ,\,\,  s\in\R  ,
\end{align} 
where the infimum extends over all possible
$\tilde f$ defined on ${\mathcal Q}$ such that
$\tilde f=f$ in $Q$. 

We adopt the following notations for the $L^2$ inner products on $\Omega$ and the space-time $L^2$ inner products on ${\mathcal Q}_T=(0,T)\times\Omega$:  
\begin{align}\label{inner-products}
	(\phi,\varphi):=\int_\Omega \phi(x)\varphi(x)\d x,\qquad
	[u,v]:=\int_0^T\int_{\Omega} u(t,x)v(t,x)\d x \, \d t .
\end{align} 
Moreover, we denote $w(t)=w(t,\cdot)$ for any function $w$ defined on ${\mathcal Q}_T$. The notation $1_{0<t<T}$ will denote the characteristic function of the time interval $(0,T)$, i.e. $1_{0<t<T}(t)=1$ if $t\in(0,T)$ while $1_{0<t<T}(t)=0$ if $t\notin(0,T)$.

\subsection{Properties of the finite element spaces}
\label{Sec2-2}

For any subdomain $D\subset\Omega$, we denote by $\mathring S_h(D)$ the space of functions of $\mathring S_h$  restricted to the domain $D$, and denote by $\mathring S_h^0(D)$ the subspace of $\mathring S_h(D)$ consisting of functions which equal zero outside $D$. For any given subset $D\subset\Omega$, we denote $D_d =\{x\in\Omega: {\rm dist}(x,D)\leq d\}$ for $d>0$. 

On a quasi-uniform triangulation of the domain $\Omega$, there exist positive constants $K $ and $\kappa$ such that the triangulation and the corresponding finite element space $\mathring S_h$ possess the following properties ($K$ and $\kappa$ are independent of the subset $D$ and $h$).
\medskip

\begin{enumerate}[label={\bf (P\arabic*)},ref=\arabic*]\itemsep=5pt

\item {\bf Quasi-uniformity:}

\noindent For all triangles (or tetrahedron) $\tau_l^h$ in the partition,
the diameter $h_l$ of $\tau_l^h$ and the radius $\rho_l$
of its inscribed ball satisfy
$$
K^{-1}h\leq \rho_l \leq h_l\leq Kh .
$$

\item {\bf Inverse inequality:}

\noindent If $D$ is a union of elements in the partition, then
\begin{align*}
	\|\chi\|_{W^{l,p}(D)}
	\leq K h^{-(l-k)-(d/q-d/p)}\|\chi\|_{W^{k,q}(D)} ,
	\quad\forall\,\chi\in \mathring S_h , 
\end{align*}
for $0\leq k\leq l\leq 1$ and 
$1\leq q\leq p\leq\infty$.  

\item {\bf Local approximation and superapproximation:} 

\noindent There exists an operator $I_h:H^1_0(\Omega)\rightarrow \mathring S_h$ with the following properties:

\begin{enumerate}[label=(\arabic*),ref=\arabic*]\itemsep=5pt
	\item 
	For $v\in H^{1+\alpha}(\Omega )\cap H^1_0(\Omega) $ the following estimate holds: 
	\begin{align*}
		&\|v-I_hv\|_{L^2} +h \|\nabla(v- I_hv)\|_{L^2} \leq Kh^{1+\alpha} \|v\|_{H^{1+\alpha}} 
		\quad \forall\, \alpha\in[0,1] ,
	\end{align*}

	\item 
	If $d\geq 2h$ then the value of $I_h v$ in $D$ depends only on the value of $v$ in $D_d$. 
	If $d\ge 2h$ and supp$(v)\subset \overline D$, then $I_hv\in \mathring S_h^0(D_d)$.

\item 
If $d\geq 2h$, $\omega=0$ outside $D$
and $|\partial^\beta\omega|\leq Cd^{-|\beta|}$
for all multi-index $\beta$,
then 
%
\begin{align*}
	&\psi_h\in \mathring S_h(D_d)\implies I_h(\omega\psi_h)\in \mathring S_h^0(D_d),\\
	&\|\omega\psi_h-I_h(\omega\psi_h)\|_{L^2}
	+h\|\omega\psi_h-I_h(\omega\psi_h)\|_{H^1}
	\leq K  h d^{-1}
	\|\psi_h\|_{L^2(D_d)} .
\end{align*}

\item  If $d\ge 2h$ and $\omega\equiv 1$ on $D_d$, then $I_h(\omega\psi_h)=\psi_h$ on $D$. 
\end{enumerate}

\end{enumerate}

The properties (P1)-(P3) hold for any quasi-uniform triangulation with the standard finite element spaces consisting of globally continuous piecewise polynomials of degree $r\ge 1$ (cf. \cite[Appendix]{Schatz95}).
Property (P3)-(1) and the definition \eqref{Def-HsD} imply the following local estimate for $\alpha\in[0,1]$ and $v\in H^{1+\alpha}(D_d)\cap H^1_0(\Omega)$: 
\begin{equation}\label{eq:Ih_app}
\|v-I_hv\|_{L^2(D)} +h\|v-I_hv\|_{H^1(D)} \leq Kh^{1+\alpha} \|v\|_{H^{1+\alpha}(D_d)}.
\end{equation}
%

{
In addition to $I_h$, we will also need the orthogonal $L^2$ projection $P_h \colon L^2(\Omega) \to \mathring S_h(\Omega)$ and Ritz-Projection $R_h \colon H^1_0(\Omega) \to \mathring S_h(\Omega)$, which are defined by
\begin{align*}
(P_hv,\chi)_{\Omega} &= (v,\chi)_{\Omega}, && \forall \chi\in \mathring S_h(\Omega), \\
(\nabla R_hv,\nabla \chi)_{\Omega} &= (\nabla v,\nabla \chi)_{\Omega}, && \forall \chi\in \mathring S_h(\Omega) . 
\end{align*}
For $v\in H^{1+\alpha}(\Omega)\cap H^1_0(\Omega)$ with $\alpha\in[0,1]$, the following global estimate follows from choosing $D=\Omega$ in \eqref{eq:Ih_app}: 
\begin{equation}\label{eq:Ph_app}
\|v-P_hv\|_{L^2(\Omega)} +h\|v-P_hv\|_{H^1(D)} \leq Kh^{1+\alpha} \|v\|_{H^{1+\alpha}(\Omega)}.
\end{equation}
Using \eqref{eq:Ih_app} and a duality argument, we also have the following global estimate for the Ritz projection:
\begin{equation}\label{eq:Rh_app}
\|v-R_hv\|_{L^2(\Omega)} +h^\alpha \|v-R_hv\|_{H^1(\Omega)} \leq Kh^{2\alpha} \|v\|_{H^{1+\alpha}(\Omega)} . 
\end{equation}
  }

\subsection{Green's functions}\label{sec:Green}
For any $x_0\in \tau_{l}^h$ (where $\tau_{l}^h$ is a triangle or a tetrahedron in the triangulation of $\Omega$), there exists a function $\tilde\delta_{x_0}\in C^3(\overline\Omega)$ with support in $\tau_{l}^h$ such that
\begin{align*}
\chi(x_0)=\int_{\Omega}\chi \tilde\delta_{x_0}\d x,
\quad\forall\,\chi\in \mathring S_h ,
\end{align*}
and
\begin{align}
&\|\tilde\delta_{x_0}\|_{W^{l,p}}
\leq K h^{-l-d(1-1/p)}
\quad\mbox{for}\,\,\,1\leq p\leq\infty,
\,\,\, l=0,1,2,3 ,\quad d\geq 1, \label{reg-Delta-est}\\
&\sup_{y\in\Omega} \int_\Omega |\tilde\delta_{y}(x)|\d x+
\sup_{x\in\Omega}\int_\Omega |\tilde\delta_{y}(x)|\d y \le C .
\label{reg-Delta-est-2}
\end{align}
The construction of $\tilde\delta_{x_0}$ can be found in \cite[Lemma 2.2]{Tho00}. 

Let $\delta_{x_0}$ denote the Dirac Delta function centered at $x_0$. In other words, the following relation holds: 
$$\int_\Omega\delta_{x_0}(y)\varphi (y)\d y=\varphi(x_0) \quad\forall\, \varphi\in C(\overline\Omega) . $$ 
Then the discrete Delta function 
$$
\delta_{h,x_0}:=P_h  \delta_{x_0}=P_h \tilde\delta_{x_0} 
$$ 
decays exponentially away from $x_0$ (cf. \cite[Lemma 2.3]{Tho00}): 
\begin{align}\label{Detal-pointwise}
& |\delta_{h,x_0}(x)|=|P_h \tilde\delta_{x_0}(x)|
\leq Kh^{-d} e^{-\frac{|x-x_0| }{K h }} ,
\quad \forall\, x,x_0\in\Omega  .
\end{align}

Let $G(t,x,x_0)$ denote the Green's function
of the parabolic equation, i.e. $G=G(\cdot,\cdot\, ,x_0)$ is the solution
of 
\begin{align}\label{GFdef}
\left\{\begin{array}{ll}
\partial_tG(\cdot,\cdot\, ,x_0)-\Delta G(\cdot,\cdot\, ,x_0)=0
&\mbox{in}\,\,\, (0,T]\times \Omega,\\
G(\cdot,\cdot\, ,x_0) = 0
&\mbox{on}\,\,\, (0,T]\times \partial\Omega ,\\
G(0,\cdot,x_0)= \delta_{x_0}
&\mbox{in}\,\,\,\Omega .
\end{array}\right.
\end{align}
The Green's function $G(t,x,y)$ is symmetric with respect to $x$ and $y$. It has an analytic extension to the right half-plane, satisfying the following Gaussian estimate
(cf. \cite[Lemma 2]{Davies97} and \cite[Theorem 3.4.8]{Davies89}): 
\begin{equation}
|G(z,x,y)|\leq C_\theta|z|^{-\frac{d}{2}}
e^{-\frac{|x-y|^2}{C_\theta |z|}}, 
\quad \forall\, z\in \varSigma_{\theta},\,\,\forall\, x,y\in\varOmega ,
\quad\forall\,\theta\in(0,\pi/2), 
\label{GKernelE0}
\end{equation}
where the constant $C_\theta$ depends only on $\theta$. 
Then Cauchy's integral formula says that 
\begin{equation}
\partial_t^kG(t,x,y)
=\frac{k!}{2\pi i}
\int_{|z-t|=\frac{t}{2}}\,\, \frac{G(z,x,y)}{(z-t)^{k+1}}\d z ,
\end{equation}
which further yields the following Gaussian pointwise estimate for the time derivatives of Green's function (cf. \cite[Corollary 5]{Davies97} and \cite[Appendix B with $\alpha=\beta=0$]{Geissert06}): 
\begin{align}
&|\partial_t^kG(t,x,x_0)|\leq \frac{C_k}{t^{k+d/2}} e^{-\frac{|x-x_0|^2}{C_kt}}, &&
\forall\, x,x_0\in\Omega,\,\,\, \forall\, t>0, \,\, k=0,1,2,\dots \label{GausEst1}
\end{align}

Let $\Gamma=\Gamma(\cdot,\cdot\, , x_0)$ 
be the regularized Green's function
of the parabolic equation, defined by 
\begin{align}\label{RGFdef}
\left\{\begin{array}{ll}
\partial_t\Gamma(\cdot,\cdot\, , x_0)-\Delta\Gamma(\cdot,\cdot\, , x_0)=0
&\mbox{in}\,\,\,(0,T]\times \Omega,\\
\Gamma(\cdot,\cdot\, , x_0) = 0
&\mbox{on}\,\,\, (0,T]\times \partial\Omega ,\\
\Gamma(0,\cdot ,x_0)=\tilde\delta_{x_0} 
&\mbox{in}\,\,\,\Omega,
\end{array}\right.
\end{align}
and let $\Gamma_h=\Gamma_h(\cdot,\cdot,  x_0)$ 
be the finite element approximation of $\Gamma$,
defined by
\begin{align}\label{EqGammh}
\left\{\begin{array}{ll}
(\partial_t\Gamma_h(t,\cdot,x_0),v_h)+(\nabla\Gamma_h(t,\cdot,x_0),\nabla v_h)=0 ,
&\forall \, v_h\in \mathring S_h,\,\, t\in(0,T),\\[5pt]
\Gamma_h(0,\cdot,x_0)=  \delta_{h,x_0} .
\end{array}\right.
\end{align}

The regularized Green's function can be expressed in terms of the Green's function as follows: 
\begin{align}\label{expr-Gamma}
&\Gamma(t,x,x_0)=\int_\Omega  G(t,y,x)\tilde\delta_{x_0}(y)\d y=\int_\Omega  G(t,x,y)\tilde\delta_{x_0}(y)\d y  . 
\end{align}
From the expression in \eqref{expr-Gamma} one can easily derive that the regularized Green's function $\Gamma$ also satisfies the Gaussian pointwise estimate for $k\ge 0$:
\begin{align}
&|\partial_t^k\Gamma(t,x,x_0)|\leq \frac{C_k}{t^{k+d/2}} e^{-\frac{|x-x_0|^2}{C_kt}}, &&
\forall\, x,x_0\in\Omega,\,\,\forall\, t>0\,\,\mbox{for}\,\, \max(|x-x_0|,\sqrt{t})\ge 2h .  \label{GausEstGamma}
\end{align}

\section{Weak maximum principle of semi-discrete FEM}\label{sec: semi-discrete}
\setcounter{equation}{0}

Since the problem is linear, the first part of the result, i.e., 
$
\|u_h\|_{L^\infty(0,T;L^\infty(\Omega))}
\le C\|u_h^0\|_{L^\infty(\Omega)}, 
$
follows from the pointwise stability of the semi-group established in \cite{Li19}. To establish the second part of the result, i.e., 
$
\|u_h\|_{L^\infty(0,T;L^\infty(\Omega))}
\le C\|u_h^0\|_{L^\infty(\Omega)} +
C\|g_h\|_{L^\infty(0,T;L^\infty(\partial\Omega))} , 
$
we will use the technique based on the dyadic decomposition and local energy estimates.

\subsection{Dyadic decomposition of the domain ${\mathcal Q}=(0,1)\times\Omega$}
\label{SecGF} 
In the proof of Theorem \ref{thm:semi-WMP}, we need to partition the domain ${\mathcal Q}=\Omega_T$ with $T=1$, i.e. ${\mathcal Q}=(0,1)\times\Omega$, into subdomains, and present estimates of the finite element solutions in each subdomain. The following dyadic decomposition of ${\mathcal Q}$ was introduced in \cite{Schatz98}. 

For any integer $j$, we define $d_j=2^{-j}$.  For a given $x_0\in\Omega$, 
we let $J_1=1$, $J_0=0$ and $J_*$ be an integer satisfying $2^{-J_*}= C_*h$ 
with $C_*\geq 16$ to be determined later.  
If 
\begin{align}\label{h-condition}
h<1/(4C_*)  ,
\end{align}
then 
\begin{align}
2\leq J_*=\log_2[1/(C_*h)]\leq \log_2(2+1/h) .
\end{align} 
Let  
\begin{align*}
&Q_*(x_0)=\{(x,t)\in\Omega_T: \max (|x-x_0|,t^{1/2})\leq d_{J_*}\}, 
\\
&\Omega_*(x_0)=\{x\in \Omega: |x-x_0|\leq d_{J_*}\} \, . 
\end{align*} 
We define 
\begin{align*} 
&Q_j(x_0)=\{(x,t)\in \Omega_T:d_j\leq
\max (|x-x_0|,t^{1/2})\leq2d_j\}  &&\mbox{for}\,\,\, j\ge 1,
\\
&\Omega_j(x_0)=\{x\in \Omega: d_j\leq|x-x_0|\leq2d_j\}   
&&\mbox{for}\,\,\, j\ge 1,
\\
&D_j(x_0)=\{x\in \Omega:  |x-x_0|\leq2d_j\}  
&&\mbox{for}\,\,\, j\ge 1,
\end{align*} 
and 
\begin{align*} 
&Q_0(x_0)= 
{\mathcal Q}\big\backslash\big( \cup_{j=1}^{J_*}Q_{j}(x_0)\cup Q_*(x_0)\big) , 
\\
&\Omega_0(x_0) 
=\Omega\big\backslash\big( \cup_{j=1}^{J_*}\Omega_{j}(x_0)\cup \Omega_*(x_0)\big).
\end{align*}
For $j<0$, we simply define
$Q_{j}(x_0)=\Omega_{j}(x_0)=\emptyset$.
For all integer $j\ge 0$, we define
\begin{align*}
\Omega_j'(x_0)&=\Omega_{j-1}(x_0)\cup\Omega_{j}(x_0)\cup\Omega_{j+1}(x_0),
\quad &Q_j'(x_0)=Q_{j-1}(x_0)\cup Q_{j}(x_0)\cup Q_{j+1}(x_0), \\
\Omega_j''(x_0)&=\Omega_{j-2}(x_0)\cup\Omega_{j}'(x_0)\cup\Omega_{j+2}(x_0),
\quad
&Q_j''(x_0)=Q_{j-2}(x_0)\cup Q_{j}'(x_0)\cup Q_{j+2}(x_0),\\
D_j'(x_0)&=D_{j-1}(x_0)\cup D_{j}(x_0),
 &D_j''(x_0)=D_{j-2}(x_0)\cup D_{j}'(x_0).
\end{align*}
Then we have
\begin{align}\label{decomposition}
&\Omega_T=\bigcup^{J_*}_{j=0}Q_j(x_0)\,\cup Q_*(x_0)
\quad\mbox{and}\quad
\Omega=\bigcup^{J_*}_{j=0}\Omega_j(x_0)\,\cup \Omega_*(x_0) .
\end{align}
We refer to $Q_*(x_0)$ as the "innermost" set.
We shall write $\sum_{*,j}$ when the innermost set is included and
$\sum_j$ when it is not. When $x_0$ is fixed, if there is no ambiguity, 
we simply write
$Q_j=Q_j(x_0)$, $Q_j'=Q_j'(x_0)$, $Q_j''=Q_j''(x_0)$, 
$\Omega_j=\Omega_j(x_0)$, $\Omega_j'=\Omega_j'(x_0)$ and 
$\Omega_j''=\Omega_j''(x_0)$.

We shall use the notations 
\begin{align}\label{Q-norm}
&\|v\|_{k,D}=\biggl(\int_D\sum_{|\alpha|\leq k}|\partial^\alpha v|^2\d x\biggl)^{\frac{1}{2}},\qquad
\vertiii{v}_{k,Q}=\biggl(\int_Q\sum_{|\alpha|\leq k}|\partial^\alpha v|^2\d x\d t\biggl)^{\frac{1}{2}} ,
\end{align}
for any subdomains $D\subset\Omega$ and $Q\subset (0,1)\times\Omega $.
Throughout this paper, we denote by $C$ a generic positive constant that 
is independent of $h$, $x_0$ and $C_*$ (until $C_*$ is determined 
in Section \ref{ProofLm03}). To simplify the notations, we also denote $d_*=d_{J_*}$.

\subsection{Technical lemmas}

Let $\Omega$ be a polygon in $\R^2$ or a polyhedron in $\R^3$ {\rm(}possibly nonconvex{\rm)}, and let $\mathring S_h$, $0<h<h_0$, be a family of finite element subspaces of $H^1_0(\Omega)$ consisting of piecewise polynomials of degree $r\ge 1$ subject to a quasi-uniform triangulation of the domain $\Omega$ (with mesh size $h$). 

The first two of the following three technical lemmas were proved in \cite[Lemma 4.1, Lemma 4.4]{Li19} for general polyhedron that may be nonconvex. In the case $\Omega$ is a convex polyhedron, $\alpha$ can be slightly bigger than $1$. 
We include this slightly different result for the convex case and omit the proof, which is almost the same as the proof for general polyhedra except some minor difference. 

\begin{lemma}\label{GFEst1}
{\it There exist $\alpha\in (\frac{1}{2},2]$ and $C>0$, 
	independent of $h$ and $x_0$, such that 
	the Green's function $G$ defined in 
	\eqref{GFdef} and the regularized Green's function 
	$\Gamma$ defined in \eqref{RGFdef}
	satisfy the following estimates:
	\begin{align}
		&d_j^{-4-\alpha+d/2}\|\Gamma(\cdot,\cdot,x_0)\|_{L^\infty(Q_j(x_0))}
		+d_j^{-4-\alpha}\vertiii{\nabla\Gamma(\cdot,\cdot,x_0)}_{L^2(Q_j(x_0))} \notag\\
		&
		+d_j^{-4}\vertiii{ \Gamma(\cdot,\cdot,x_0)}_{L^2H^{1+\alpha}(Q_j(x_0))}  +d_j^{-2}\vertiii{ \partial_{t} 
			\Gamma(\cdot,\cdot,x_0)}_{L^2H^{1+\alpha}(Q_j(x_0))} \notag\\
		& 
		+\vertiii{ \partial_{tt}\Gamma(\cdot,\cdot,x_0) }_{L^2H^{1+\alpha}(Q_j(x_0))}
		\leq Cd_j^{-d/2-4-\alpha}, \label{GFest01}\\[10pt]
		&\|G(\cdot,\cdot,x_0) \|_{L^{\infty}H^{1+\alpha}(\cup_{k\leq
				j}Q_k(x_0))}
		+d_j^2\|\partial_tG(\cdot,\cdot,x_0) \|_{L^{\infty}H^{1+\alpha}(\cup_{k\leq
				j}Q_k(x_0))}\leq Cd_j^{-d/2-1-\alpha} \label{GFest03} .
	\end{align}
	
	In the case $d=2$, $\alpha\in(\frac12,1]$. In the case $d=3$ and $\Omega$ is convex, $\alpha\in(1,2]$. 
}
\end{lemma}

In the rest of this paper, we denote by $\alpha$ a number depending on which case we are working on, i.e., $\alpha\in(\frac12,1]$ in the case $d=2$, and $\alpha\in(1,2]$ in the case $d=3$ and $\Omega$ is convex. 

\begin{lemma}\label{LemGm2}
{\it
	The functions $\Gamma_h(t,x,x_0)$, $\Gamma(t,x,x_0)$ and 
	$F(t,x,x_0):=\Gamma_h(t,x,x_0)-\Gamma(t,x,x_0)$ satisfy 
	\begin{align}
		&\sup_{t\in(0,\infty)}\,\left(\|\Gamma_h(t,\cdot, x_0)\|_{L^1(\Omega)}  
		+t\|\partial_{t}\Gamma_h(t,\cdot, x_0)\|_{L^1(\Omega)} \right) \leq C , 
		\label{L1Gammh}\\ 
		&\sup_{t\in(0,\infty)}\,\left(\|\Gamma(t,\cdot, x_0)\|_{L^1(\Omega)}  
		+t\|\partial_{t}\Gamma(t,\cdot, x_0)\|_{L^1(\Omega)} \right) \leq C , 
		\label{L1Gammh-2}\\ 
		&\|\partial_tF(\cdot,\cdot ,x_0)\|_{L^1((0,\infty)\times\Omega)}  
		+\|t\partial_{tt}F(\cdot,\cdot , x_0)\|_{L^1((0,\infty)\times\Omega)} \leq C ,
		\label{L1Ft}\\
		&
		\|\partial_t\Gamma_h(t,\cdot, x_0)\|_{L^1}\leq Ce^{-\lambda_0t} , \qquad\forall\, t\ge 1, \label{L1Gammatx0}
	\end{align} 
	where the constants $C$ and $\lambda_0$ are independent of $h$. 
}
\end{lemma}

\begin{lemma}\label{lemma:three}
	The following results hold:
	\begin{enumerate}
		\item
		$\| \Gamma(t) \|_{L^2 (\frac14, \infty; H^{1+\alpha}(\Omega))} + \| \partial_t \Gamma(t) \|_{L^2 (\frac14, \infty; H^{1+\alpha}(\Omega))} \le C$.
		
		\item
		$(-\Delta_h)^{-1}:(\mathring S_h(\Omega), \|\cdot\|_{L^1(\Omega)}) \rightarrow (\mathring S_h(\Omega), \|\cdot\|_{L^2(\Omega)})$ is continuous.
		
		\item
		$\| \nabla (\Gamma_h-\Gamma) \|_{L^1(0,1; L^1(\Omega))}\le C h^{\frac12}$.
	\end{enumerate}
\end{lemma} 
\begin{proof}
	To show the first result, we recall the estimate in \cite[p. 27]{Li19} which can be proved by the standard energy method. Here we additionally allow $t\geq \frac14$ instead of $t\geq 1$ in \cite[p. 27]{Li19}:
	\begin{align}
		&\| \partial_{t} \Gamma(t,\cdot, x_0) \|_{L^2}^2
		+
		\| \partial_{t} \Gamma_h(t,\cdot, x_0) \|_{L^2}^2
		+
		\| \partial_{tt} \Gamma(t,\cdot, x_0) \|_{L^2}^2
		+
		\| \partial_{tt} \Gamma_h(t,\cdot, x_0) \|_{L^2}^2 \notag\\
		&\leq C e^{-\lambda_0 (t - \frac14)} . \notag
	\end{align}
	Therefore from elliptic regularity theory, we get the following boundedness results
	\begin{align}
		&\| \Gamma(t) \|_{L^2 (\frac14, \infty; H^{1+\alpha})} \leq C \| \Delta \Gamma(t) \|_{L^2 (\frac14, \infty; L^{2})} = C \| \partial_{t} \Gamma(t) \|_{L^2 (\frac14, \infty; L^{2})} \leq C , \notag\\
		&\| \partial_t \Gamma(t) \|_{L^2 (\frac14, \infty; H^{1+\alpha})} \leq C \| \Delta \partial_t \Gamma(t) \|_{L^2 (\frac14, +\infty; L^{2})} = C \| \partial_{tt} \Gamma(t) \|_{L^2 (\frac14, +\infty; L^{2})} \leq C . \notag
	\end{align}
	
	To prove the second one, for any given $f_h\in \mathring S_h(\Omega)$, we define $u_h = (-\Delta_h)^{-1} f_h$ and $u\in H_0^1(\Omega)$ to be the solution of the elliptic equation
	\begin{align}
		-\Delta u = f_h . \notag
	\end{align}
	By the definition we have the relation $u_h = R_h u$. For $d = 2, 3$, fix $p=2^*=\frac{2d}{d+2}$ whose H\"older conjugate is $p^\prime = \frac{2d}{d-2}$, and we apply the Sobolev embedding $W^{1,p} \hookrightarrow L^2$, the stability of Ritz projection $R_h$ (see \cite[equation (2.5)]{Leykekhman_Li_2021}), the elliptic regularity theory on corner domain (see \cite[Theorem 3.2, Corollary 3.10, Corollary 3.12]{Dauge92}) and the embedding $W^{1,p^\prime} \hookrightarrow L^\infty$ successively to derive
	\begin{align}
		\| (-\Delta_h)^{-1} f_h \|_{L^2} 
		=\| u_h \|_{L^2} 
		\leq C \| u_h \|_{W^{1, p}} 
		\leq C \| u \|_{W^{1, p}}
		\leq C \| f_h \|_{W^{-1, p}}
		\leq C \| f_h \|_{L^1} . \notag
	\end{align}
	For the third result, we modify the quantity $\mathcal K$ in the proof of \cite[Lemma 4.4]{Li19} to be
	\begin{align}
		{\mathcal K}  :&=
		\sum_{j} d_{j}^{1+d/2}\left(d_j^{-1/2} h^{-1/2}\vertiii{F}_{1,Q_j}
		+\vertiii{\partial_tF}_{Q_j}
		+d_{j}\vertiii{\partial_tF}_{1,Q_j}+d_{j}^2\vertiii{\partial_{tt}F}_{Q_j}\right)  , \notag
	\end{align} 
	and then the same proof will show $\mathcal K \leq C$. Sketch of the proof: The lowest powers of $h$ in the consistency part in the proof of \cite[Lemma 4.4]{Li19} are $h^\alpha$ (see \cite[equation (5.14) and (5.15)]{Li19}) and $h^{1+\alpha-d/2}$ (see \cite[equation (5.32)]{Li19}). Therefore if we multiply $h^{-1/2}$ in the front, we get $h^{-1/2+\alpha} + h^{1/2+\alpha - d/2} \leq C$. For the stability part, as \cite[equation (5.32)]{Li19}, we will get an additional power of $h^{-1/2}h^{1+\alpha}=h^{1/2+\alpha}$ in front of $\vertiii{ F_t }_{Q_i}$ and an additional power of $h^{\alpha}$ in front of $\vertiii{ F}_{1, Q_i}$ (note that according to the definition of $\mathcal K$, there already exists $h^{-1/2}$ in front of $\vertiii{ F}_{1, Q_i}$ on the left hand side). Such additional positive powers of $h$ ensure the stability of $\mathcal K$.
\end{proof}

\subsection{Reduction of the problem}\label{sec:proof2.3}
$\,$
To simplify the notation, we relax the dependence of the Green's functions on $x_0$ and denote 
$$
\Gamma_h(t)=\Gamma_h(t,\cdot,  x_0),\qquad
\Gamma(t)=\Gamma(t,\cdot,  x_0)\quad\mbox{and}\quad
F(t)=\Gamma_h(t)-\Gamma(t) . 
$$ 
If $u_h$ is the finite element solution of
\begin{align}
\left\{
\begin{aligned}
	&(\partial_tu_h,v_h) + (\nabla u_h,\nabla v_h) = 0 &&\forall\, v_h\in \mathring S_h,\,\,\,\mbox{for}\,\,\,t\in (0,T] , \\
	&u_h = g_h &&\mbox{on}\,\,\,[0,T]\times\partial\Omega, \\
	&u_h(0)=0&&\mbox{in}\,\,\,\Omega , 
\end{aligned}
\right.
\end{align}
then substituting $v_h=\Gamma_h(s-t)$, with $0<t<s\le T$, and integrating the result for $t\in(0,s)$ yield 
\begin{align}
u_h(s,x_0) 
&= -\int_0^s \Big[ (u_h,\partial_t\Gamma_h(s-t)) + (\nabla u_h,\nabla \Gamma_h(s-t)) \Big]\d t \notag\\
&= - \int_0^s \Big[ (u_h-w_h,\partial_t\Gamma_h(s-t)) + (\nabla (u_h-w_h),\nabla \Gamma_h(s-t)) \Big]\d t  ,
\end{align}
which holds for all $w_h\in L^2(0,T;\mathring S_h)$, 
where the last inequality is a consequence of \eqref{EqGammh}. 
Let $w_h\in\mathring S_h $ be the finite element function which equals $u_h$ at interior nodes and equals $0$ on $\partial\Omega$. Then
\begin{align}
u_h(s,x_0) 
&= -\int_0^s \Big[ (\tilde{g}_h,\partial_t\Gamma_h(s-t)) + (\nabla \tilde{g}_h,\nabla \Gamma_h(s-t)) \Big]\d t  ,
\end{align}
where $\tilde g_h$ is the finite element solution which equals $g_h$ on $\partial\Omega$ and equals zero at the interior nodes of $\Omega$.
Let $D_h=\{x\in\Omega:{\rm dist}(x,\partial\Omega)\le h\}$. Then 
\begin{align}\label{u_h(s,x_0)}
u_h(s,x_0) 
&= -\int_0^s \Big[ (\tilde{g}_h,\partial_t\Gamma_h(s-t))_{D_h} + (\nabla \tilde{g}_h,\nabla \Gamma_h(s-t))_{D_h} \Big]\d t .
\end{align}
Let $\tilde u$ be the solution of the PDE problem 
\begin{align}\label{PDE-tilde-u}
\left\{
\begin{aligned}
	&\partial_t\tilde u - \Delta \tilde u= 0 &&\mbox{for}\,\,\,t\in (0,T] , \\
	&\tilde u = g_h &&\mbox{on}\,\,\,[0,T]\times\partial\Omega , \\
	&\tilde u(0)=0&&\mbox{in}\,\,\,\Omega .
\end{aligned}
\right.
\end{align}
By the maximum principle of parabolic PDEs (see \cite[Theorem 6.2.6]{Are11} and \cite[Chapter II]{Lie96}), we have 
\begin{align}
\|\tilde u\|_{L^\infty(0,T;L^\infty)} \le \|g_h\|_{L^\infty(0,T;L^\infty(\partial\Omega))} . 
\end{align}
Testing the first equation of \eqref{PDE-tilde-u} by $\Gamma(s-t)$ and integrating the result for $t\in(0,s)$, we obtain 
\begin{align}
(\tilde u(s,\cdot) ,\tilde\delta_{x_0}) 
&= -\int_0^s \Big[ (\tilde u,\partial_t\Gamma(s-t)) + (\nabla \tilde u,\nabla \Gamma(s-t)) \Big]\d t \notag\\
&= -\int_0^s \Big[ (\tilde u-w,\partial_t\Gamma(s-t)) + (\nabla (\tilde u-w),\nabla \Gamma(s-t)) \Big]\d t , \end{align}
which holds for all $w\in L^2(0,T;H^1_0(\Omega) )$. By choosing $w=\tilde u-\tilde{g}_h$ we obtain 
\begin{align}\label{u(s,x_0)}
(\tilde u(s,\cdot) ,\tilde\delta_{x_0}) 
&= -\int_0^s \Big[ (\tilde{g}_h,\partial_t\Gamma(s-t)) + (\nabla \tilde{g}_h,\nabla \Gamma(s-t)) \Big]\d t \notag \\
&= -\int_0^s \Big[ (\tilde{g}_h,\partial_t\Gamma(s-t))_{D_h} + (\nabla \tilde{g}_h,\nabla \Gamma(s-t))_{D_h} \Big]\d t .
\end{align}
Subtracting \eqref{u(s,x_0)} from \eqref{u_h(s,x_0)}, we obtain
\begin{align}\label{u_h-expr}
u_h(s,x_0) 
=(\tilde u(s,\cdot) ,\tilde\delta_{x_0})
&- \int_0^s  \big(\tilde{g}_h,\partial_t\Gamma_h(s-t) - \partial_t\Gamma(s-t) \big)_{D_h} \d t \notag\\
&- \int_0^s  \big(\nabla \tilde{g}_h,\nabla (\Gamma_h(s-t)-\Gamma(s-t)) \big)_{D_h} \d t . 
\end{align}
From \eqref{L1Ft} we know that 
$$
\int_0^s \int_\Omega |\partial_t\Gamma_h(s-t)-\partial_t\Gamma(s-t)| \d x\d t \le C ,
$$
which implies that 
\begin{align}\label{u_h-expr-2}
&\hspace{-10pt} |u_h(s,x_0) | \notag\\
\le 
&\, \|\tilde u(s) \|_{L^\infty} \| \tilde\delta_{x_0}\|_{L^1} 
+ C\|\tilde{g}_h\|_{L^\infty(0,s;L^\infty)} \notag  + \bigg|\int_0^s  \big(\nabla \tilde{g}_h,\nabla (\Gamma_h(s-t)-\Gamma(s-t)) \big)_{D_h} \d t \bigg| \notag \\
\le&\,
C\|g_h\|_{L^\infty(0,s;L^\infty(\partial\Omega))}
+ \|\tilde {g}_h\|_{L^\infty(0,s;L^\infty(\Omega))} Ch^{-1} \|\nabla F\|_{L^1(D_h^s)} \notag\\
\le&\,
C\|g_h\|_{L^\infty(0,s;L^\infty(\partial\Omega))}
(1+ h^{-1} \|\nabla F\|_{L^1(D_h^s)} ) ,
\end{align}
where $D_h^s=(0,s)\times D_h$. 
If the following estimate can be proved:
\begin{align}\label{grad-Fh-estimate}
h^{-1} \|\nabla F\|_{L^1((0,\infty)\times D_h)} \le C , 
\end{align}
then substituting \eqref{grad-Fh-estimate} into \eqref{u_h-expr-2} immediately yields
\begin{align}
|u_h(s,x_0) |
\le&\,
C\|g_h\|_{L^\infty(0,s;L^\infty(\partial\Omega))} . 
\end{align}
Since $s$ and $x_0$ can be arbitrary, this proves the desired result of Theorem \ref{thm:semi-WMP}. 

It remains to prove \eqref{grad-Fh-estimate}.

\subsection{Proof of  \eqref{grad-Fh-estimate}}
\label{ProofLm03}

We need to use the following local energy error estimate for finite element solutions of parabolic equations, which was proved in \cite[Lemma 5.1]{Li19}.

\begin{lemma}\label{LocEEst} 
{\it 
	Suppose that $\phi\in L^2(0,T;H^1_0(\Omega))\cap H^1(0,T;L^2(\Omega))$ and 
	$\phi_h\in H^1(0,T; \mathring S_h)$ satisfy the equation 
	\begin{align}  
		(\partial_t(\phi-\phi_h),\chi)
		+(\nabla (\phi-\phi_h),\nabla \chi) =0, 
		\quad\forall\, \chi\in  \mathring S_h, \,\, \mbox{a.e.}\,\, t>0, 
		\label{p32}  
	\end{align} 
	with $\phi(0)=0$ in $\Omega_j''$. 
	Then 
	\begin{align} 
		\hspace{-8pt}
		&\vertiii{\partial_t(\phi-\phi_h)}_{Q_j} 
		+ d_j^{-1}\vertiii{\phi-\phi_h}_{1,Q_j} \notag\\ 
		& \leq 
		C\epsilon^{-3}\big(I_j(\phi_{h}(0))+X_j(I_h\phi-\phi) 
		+d_j^{-2} \vertiii{\phi-\phi_h}_{Q_j'}\big) \notag\\
		&\quad 
		+ \bigg[C\bigg(\frac{h}{d_j}\bigg)^{\frac12} + \epsilon_*\bigg] \big(
		\vertiii{\partial_t(\phi-\phi_h)}_{Q_j'}
		+d_j^{-1}\vertiii{\phi-\phi_h}_{1,Q_j'}\big), 
		\label{LocEngErr} 
	\end{align} 
	where $\epsilon_* = \epsilon + \epsilon^{-1}/C_*$ and 
	\begin{align*}
		&I_j(\phi_{h}(0))=\|\phi_{h}(0)\|_{1,\Omega_j'} + d_j^{-1}\|\phi_{h}(0)\|_{\Omega_j'} \, ,\\[5pt]
		&X_j(I_h\phi-\phi)=
		d_j\vertiii{\partial_t(I_h\phi-\phi)}_{1,Q_j'}
		+\vertiii{\partial_t(I_h\phi-\phi)}_{Q_j'} 
		\notag \\
		&\qquad\qquad\qquad\,\, 
		+d_j^{-1}\vertiii{I_h\phi-\phi}_{1,Q_j'}+
		d_j^{-2}\vertiii{I_h\phi-\phi}_{Q_j'}  \, .
	\end{align*}
	The positive constant $C$ is independent of $h$, $j$ and $C_*$; the norms $\vertiii{\cdot}_{k,Q_j'}$ and $\vertiii{\cdot}_{k,\Omega_j'}$ are defined in \eqref{Q-norm}. 
	
}
\end{lemma}
In the rest of this section, we apply Lemma \ref{LocEEst} to estimate $h^{-1} \|\nabla F\|_{L^1((0,\infty)\times D_h)}$. The estimation consists of two parts: The first part concerns estimates for $t \in (0, 1)$, and the second part concerns estimates for $t\ge 1$, which is a consequence of the smoothing property of parabolic equations. 

\medskip

{\it Part I.}$\,\,\,$ 
First, we present estimates in the domain ${\mathcal Q}=(0,1)\times\Omega$ with the restriction $h<1/(4C_*) $; see \eqref{h-condition}. In this case, the basic energy estimate gives 
\begin{align} 
&\|\partial_t\Gamma\|_{L^2({\mathcal Q})}
+\|\partial_t\Gamma_h\|_{L^2({\mathcal Q})}
\leq C(\|\Gamma(0)\|_{H^1} +\|\Gamma_h(0)\|_{H^1})
\leq Ch^{-1-d/2} ,  \label{EstGamma-Q1} \\
&\|\Gamma\|_{L^\infty L^2({\mathcal Q})} 
+\|\Gamma_h\|_{L^\infty L^2({\mathcal Q})} 
\leq C(\|\Gamma(0)\|_{L^2} +\|\Gamma_h(0)\|_{L^2})
\leq Ch^{-d/2} ,\\
&\|\nabla\Gamma\|_{L^2({\mathcal Q})} +\|\nabla\Gamma_h\|_{L^2({\mathcal Q})}
\leq C(\|\Gamma(0)\|_{L^2} +\|\Gamma_h(0)\|_{L^2})
\leq Ch^{-d/2} , \label{EstGamma-Q3} \\
&\|\partial_{tt}\Gamma\|_{L^2({\mathcal Q})}
+\|\partial_{tt}\Gamma_h\|_{L^2({\mathcal Q})}
\leq C(\|\Delta\Gamma(0)\|_{H^1} +\|\Delta_h\Gamma_h(0)\|_{H^1})
\leq Ch^{-3-d/2} ,  \label{EstGamma-Q4}  \\
&\|\nabla\partial_t\Gamma\|_{L^2({\mathcal Q})}
+\|\nabla\partial_t\Gamma_h\|_{L^2({\mathcal Q})}
\leq C(\|\Delta\Gamma(0)\|_{L^2} +\|\Delta_h\Gamma_h(0)\|_{L^2})
\leq Ch^{-2-d/2} , \label{EstGamma-Q5}
\end{align} 
where we have used \eqref{reg-Delta-est} and \eqref{Detal-pointwise} to estimate $\Gamma(0)$ and $\Gamma_h(0)$, respectively. Hence, we have 
\begin{align}\label{L2Gamma-Q}
\vertiii{\Gamma}_{Q_*} +\vertiii{\Gamma_h}_{Q_*} 
\le Cd_*\|\Gamma\|_{L^\infty L^2(Q_*)}+Cd_*\|\Gamma_h\|_{L^\infty L^2(Q_*)}
\le Cd_* h^{-d/2} \le CC_*h^{1-d/2} .
\end{align} 
Since the volume of $Q_j$ is $Cd_j^{2+d}$, we can decompose 
$$ \|\partial_tF\|_{L^1(D_h^T)}+\|t\partial_{tt}F\|_{L^1(D_h^T)} + h^{-1} \|\nabla F\|_{L^1(D_h^T)} $$ in the following way:
\begin{align}\label{Bd31K2}
&\|\partial_tF\|_{L^1(D_h^T)}+\|t\partial_{tt}F\|_{L^1(D_h^T)} + h^{-1} \|\nabla F\|_{L^1(D_h^T)}  \notag\\
&\leq  \|\partial_{t} F\|_{L^1(Q_*\cap D_h^T)}
+ \|t\partial_{tt} F\|_{L^1(Q_*\cap D_h^T)} 
+ h^{-1} \|\nabla F\|_{L^1(Q_*\cap D_h^T)} \notag\\
&\quad\, +\sum_{j}\big(\|\partial_{t} F\|_{L^1(Q_j\cap D_h^T)}
+ \|t\partial_{tt} F\|_{L^1(Q_j\cap D_h^T)}
+ h^{-1} \|\nabla F\|_{L^1(Q_j\cap D_h^T)} \big)  \notag\\
&\leq Cd_{*}^{1+\frac{d-1}{2}}h^{\frac12}\big(\vertiii{\partial_tF}_{Q_*\cap D_h^T} 
+ d_*^2\vertiii{\partial_{tt} F}_{Q_*\cap D_h^T}
+ h^{-1} \vertiii{F}_{1,Q_*\cap D_h^T}  \big)  \notag\\
&\quad\, 
+\sum_{j}Cd_{j}^{1+\frac{d-1}{2}}h^{\frac12}
\big(\vertiii{\partial_tF}_{Q_j\cap D_h^T} 
+d_j^2\vertiii{\partial_{tt} F}_{Q_j\cap D_h^T}
+ h^{-1} \vertiii{F}_{1,Q_j\cap D_h^T}  \big) \notag\\
&\leq CC_*^{\frac{5+d}{2}}
+ {\mathscr K} ,
\end{align} 
where we have used \eqref{EstGamma-Q1}, \eqref{EstGamma-Q3} and \eqref{EstGamma-Q4} to estimate 
$$
Cd_{*}^{1+\frac{d-1}{2}}h^{\frac12} \big(\vertiii{\partial_tF}_{Q_*\cap D_h^T} 
+d_*^2\vertiii{\partial_{tt} F}_{Q_*\cap D_h^T}+ h^{-1} \vertiii{F}_{1,Q_*\cap D_h^T} \big)  , 
$$ 
and introduced the notation 
\begin{align}\label{express-K}
{\mathscr K}  :&=
\sum_{j} \bigg(\frac{h}{d_j}\bigg)^{-\frac12} d_{j}^{1+\frac{d}{2}}(d_j^{-1}\vertiii{F}_{1,Q_j}
+\vertiii{\partial_tF}_{Q_j}
+d_{j}\vertiii{\partial_tF}_{1,Q_j}+d_{j}^2\vertiii{\partial_{tt}F}_{Q_j})  . 
\end{align} 

It remains to estimate ${\mathscr K} $. To this end, 
we set ``$\phi_h=\Gamma_h$, $\phi=\Gamma$, $\phi_h(0)=P_h\tilde\delta_{x_0}$ and $\phi(0)=\tilde\delta_{x_0}$'' 
and ``$\phi_h=\partial_t\Gamma_h$, $\phi=\partial_t\Gamma$, $\phi_h(0)=\Delta_hP_h\tilde\delta_{x_0}$ and $\phi(0)=\Delta\tilde\delta_{x_0}$''
in Lemma \ref{LocEEst}, respectively.
Then we obtain 
\begin{align}\label{F1tE}
d_j^{-1}\vertiii{F}_{1,Q_j}+\vertiii{\partial_tF}_{Q_j} 
&\le 
C\epsilon^{-3}(\widehat{I_j}+\widehat{X_j}+ d^{-2}_j\vertiii{ F }_{Q'_j} )  \\
&\quad + \bigg[C\bigg(\frac{h}{d_j}\bigg)^{\frac12} +\epsilon_*\bigg] \big( d_j^{-1}\vertiii{F}_{1,Q_j'}
+ \vertiii{\partial_{t}F }_{Q_j'}\big) \qquad \quad\,\, \notag
\end{align}
and
\begin{align}\label{F1ttE}
d_{j}\vertiii{\partial_tF}_{1,Q_j}+d_{j}^2\vertiii{\partial_{tt}F}_{Q_j} 
&\le 
C \epsilon^{-3} (\overline {I_j}+\overline{X_j} + \vertiii{ \partial_tF }_{Q'_j} )  \\
&\quad 
+ \bigg[C\bigg(\frac{h}{d_j}\bigg)^{\frac12} +\epsilon_* \bigg] \big(  d_j\vertiii{\partial_{t}F}_{1,Q_j'}
+d_{j}^2\vertiii{\partial_{tt}F }_{Q_j'} \big) 
\, , \notag
\end{align}
respectively. 
By using \eqref{eq:Ih_app} (local interpolation error estimate), \eqref{Detal-pointwise} (exponential decay of $ P_h\tilde \delta_{x_0}$) and Lemma \ref{GFEst1} (estimates of regularized Green's function), we have 
\begin{align}
\widehat{I_j}&=\|P_h\tilde\delta_{x_0}\|_{1,\Omega_j'}
+h^{-1}\|P_h\tilde\delta_{x_0}\|_{\Omega_j'} 
\leq Ch^2d_{j}^{-3-d/2},  \label{EstIj}  \\[5pt]
\widehat{X_j}&=
d_j\vertiii{\partial_t(I_h\Gamma-\Gamma)}_{1,Q_j'}
+\vertiii{\partial_t(I_h\Gamma-\Gamma)}_{Q_j'}  \notag\\
&\quad\, 
+d_j^{-1}\vertiii{I_h\Gamma-\Gamma}_{1,Q_j'}
+ d_j^{-2}\vertiii{I_h\Gamma-\Gamma}_{Q_j'}  \notag\\
&
\leq Cd_j h^{\alpha} \vertiii{\partial_{t}\Gamma}_{L^2H^{1+\alpha}(Q_j'')}
+ d_j ^{-1} h^{\alpha}\vertiii{\Gamma}_{L^2H^{1+\alpha}(Q_j'')} \notag\\
&
\leq C h^{\alpha} d_j^{-\alpha-1-d/2}   
\end{align}
and
\begin{align}
\overline{I_j}&=d_{j}^{2}\|\Delta_hP_h\tilde\delta_{x_0}\|_{1,\Omega_j'}
+d_j \|\Delta_hP_h\tilde\delta_{x_0}\|_{\Omega_j'}
\leq Ch^2d_{j}^{-3-d/2},\\[5pt]
\overline{X_j}&=
d_j^3\vertiii{I_h\partial_{tt}\Gamma-\partial_{tt}\Gamma }_{1,Q_j'}
+d_j^2\vertiii{ I_h\partial_{tt}\Gamma-\partial_{tt}\Gamma}_{Q_j'} \notag\\
&\quad\, 
+d_j \vertiii{I_h\partial_{t}\Gamma-\partial_{t}\Gamma}_{1,Q_j'}+
\vertiii{I_h\partial_{t}\Gamma-\partial_{t}\Gamma}_{Q_j'}  \notag\\
&
\leq ( d_j^3 h^{\alpha}+ d_j^{ 2}h^{1+\alpha})\vertiii{\partial_{tt}\Gamma}_{L^2H^{1+\alpha}(Q_j'')}
+(d_j h^{\alpha}+ h^{1+\alpha})\vertiii{\partial_{t}\Gamma}_{L^2H^{1+\alpha}(Q_j'')}  \notag\\
&
\leq C h^{\alpha} d_j^{-\alpha-1-d/2}   .
\label{ovXj}
\end{align}
By substituting \eqref{F1tE}-\eqref{ovXj} into the expression of ${\mathscr K}$ in \eqref{express-K}, we have
\begin{align} 
{\mathscr K} 
&=\sum_{j} \bigg(\frac{h}{d_j}\bigg)^{-\frac12} d_{j}^{1+\frac{d}{2}} (d_j^{-1}\vertiii{F}_{1,Q_j}
+\vertiii{\partial_tF}_{Q_j}
+d_{j}\vertiii{\partial_tF}_{1,Q_j}
+d_{j}^2\vertiii{\partial_{tt}F}_{Q_j})  \notag\\
&\leq C \sum_{j} \bigg(\frac{h}{d_j}\bigg)^{-\frac12} d_{j}^{1+\frac{d}{2}}
\epsilon^{-3}
\big(h^2d_j^{-3-d/2}+h^{\alpha}d_j^{-\alpha-1-d/2}
+d_j^{-2}\vertiii{F}_{Q'_j} \big) \notag\\
&\quad 
+ \sum_{j} \bigg[C + \epsilon_* \bigg(\frac{h}{d_j}\bigg)^{-\frac12} \bigg] 
d_{j}^{1+\frac{d}{2}} 
(d_j^{-1}\vertiii{F}_{1,Q_j'} +\vertiii{\partial_tF}_{Q_j'}
+d_{j}\vertiii{\partial_tF}_{1,Q_j'}
+d_{j}^2\vertiii{\partial_{tt}F}_{Q_j'} ) \notag\\
&\leq C 
+C\epsilon^{-3}  \sum_{j} \bigg(\frac{h}{d_j}\bigg)^{-\frac12} d_{j}^{-1+\frac{d}{2}} \vertiii{F}_{Q_j'}  \notag\\
&\quad 
+ \sum_{j} \bigg[C + \epsilon_* \bigg(\frac{h}{d_j}\bigg)^{-\frac12} \bigg]  d_{j}^{1+\frac{d}{2}}
(d_j^{-1}\vertiii{F}_{1,Q_j'} +\vertiii{\partial_tF}_{Q_j'}
+d_{j}\vertiii{\partial_tF}_{1,Q_j'}
+d_{j}^2\vertiii{\partial_{tt}F}_{Q_j'} )  .
\end{align}

Since $\vertiii{F}_{Q_j'} \le C(\vertiii{F}_{Q_{j-1}} + \vertiii{F}_{Q_{j}} + \vertiii{F}_{Q_{j+1}})$, 
we can convert the $Q_j'$-norm in the inequality above to the $Q_j$-norm:
\begin{align}\label{slkll}
{\mathscr K} &\leq C 
+ C\epsilon^{-3}\bigg(\frac{h}{d_*}\bigg)^{-\frac12} d_{*}^{-1+\frac{d}{2}}  \vertiii{F}_{Q_*} 
+ C\epsilon^{-3} \sum_{j} \bigg(\frac{h}{d_j}\bigg)^{-\frac12} d_{j}^{-1+\frac{d}{2}}  \vertiii{F}_{Q_j} \notag\\ 
&\quad 
+\sum_{j}\bigg[C + \epsilon_* \bigg(\frac{h}{d_j}\bigg)^{-\frac12} \bigg]  d_*^{1+\frac{}{2}}
(d_*^{-1}\vertiii{F}_{1,Q_*} +\vertiii{\partial_tF}_{Q_*}
+d_{*}\vertiii{\partial_tF}_{1,Q_*}
+d_{*}^2\vertiii{\partial_{tt}F}_{Q_*} )  \notag\\ 
&\quad 
+\sum_{j}\bigg[C + \epsilon_* \bigg(\frac{h}{d_j}\bigg)^{-\frac12} \bigg]  d_{j}^{1+\frac{d}{2}}
(d_j^{-1}\vertiii{F}_{1,Q_j} +\vertiii{\partial_tF}_{Q_j}
+d_{j}\vertiii{\partial_tF}_{1,Q_j}
+d_{j}^2\vertiii{\partial_{tt}F}_{Q_j} ) \notag\\
&\leq 
C 
+C\epsilon^{-3} C_*^{\frac{d-1}{2}}+ 
\sum_{j}C\epsilon^{-3} \bigg(\frac{h}{d_j}\bigg)^{-\frac12} d_{j}^{-1+\frac{d}{2}} \vertiii{F}_{Q_j} 
+ \bigg[C + \epsilon_*\bigg(\frac{h}{d_*}\bigg)^{-\frac12} \bigg]  C_*^{3+\frac{d}{2}} \notag\\
&\quad\, 
+ (\epsilon_* +C C_*^{-1/2}) {\mathscr K}. 
\end{align}
where we have used $d_j\ge C_*h$ and \eqref{EstGamma-Q1}-\eqref{EstGamma-Q5} to estimate 
$$
\mbox{$\vertiii{F}_{1,Q_*}$,\,\, $\vertiii{\partial_tF}_{Q_*}$,\,\, $\vertiii{\partial_tF}_{1,Q_*}$\,\, and\,\, 
$\vertiii{\partial_{tt}F}_{Q_*}$}
$$
and used the expression of ${\mathscr K}$ in \eqref{express-K} to bound the terms involving $Q_j$.  
Since $\epsilon_* = \epsilon + \epsilon^{-1}/C_*$, we can make $\epsilon_*$ sufficiently small by first choosing $\epsilon$ small enough and then choosing $C_*$ large enough ($\epsilon$ can be fixed now and $C_*$ will be determined later). Then the last term on the right-hand side of \eqref{slkll} can be absorbed by the left-hand side. Therefore, we obtain
\begin{align} \label{K-L2F}
{\mathcal K}
&\leq
C +C C_*^{3+\frac{d}{2}}+
\sum_{j}\bigg(\frac{h}{d_j}\bigg)^{-\frac12} d_{j}^{-1+\frac{d}{2}}  \vertiii{F}_{Q_j} .
\end{align}

It remains to estimate $\vertiii{F}_{Q_j}$. This was estimated in \cite[inequality (5.32)]{Li19} as 
\begin{align}\label{FL2Qj} 
\vertiii{F}_{Q_j} 
&\leq Ch^{1+\alpha-d/2}d_j^{-\alpha} 
+C\sum_{*,i}(h^{1+\alpha}\vertiii{F_t}_{Q_i} 
+ h^{\alpha} \vertiii{F}_{1,Q_i})
d_i^{1-\alpha} \bigg(\frac{\min (d_i ,d_j )}{\max (d_i ,d_j )}\bigg)^{\alpha} .
\end{align}
Note that
\begin{align}\label{lalnu}
\sum_j \bigg(\frac{h}{d_j}\bigg)^{-\frac12} d_{j}^{-1+\frac{d}{2}} 
\bigg(\frac{\min (d_i ,d_j )}{\max (d_i ,d_j )}\bigg)^{\alpha}
\leq 
\left\{
\begin{aligned}
	& C \bigg(\frac{h}{d_i}\bigg)^{-\frac12}  &&\mbox{if $d=2$ and $r\ge 1$} , \\
	& C\bigg(\frac{h}{d_i}\bigg)^{-\frac12}  d_i^{\frac12} &&\mbox{if $d=3$ and $r\ge 2$} ,
\end{aligned}\right.
\end{align}
where the last inequality uses the assumption that in the case $d=3$ the domain is convex so that $\alpha>1$ can be chosen; see (P3) in Section \ref{Sec2-2}. 
By substituting \eqref{FL2Qj}--\eqref{lalnu} into \eqref{K-L2F} we obtain 
\begin{align*}
{\mathscr K} &\leq 
C+CC_*^{3+d/2}+
\sum_{j} Cd_{j}^{-1+\frac{d}{2}} \vertiii{F}_{Q_j} \notag\\
&\leq C+CC_*^{3+d/2}
+C \sum_j \left(\frac{h}{d_j}\right)^{\alpha+\frac12-\frac{d}{2}} 
\qquad\mbox{here we substitute \eqref{FL2Qj} }\\
&\quad
+C\sum_j \bigg(\frac{h}{d_j}\bigg)^{-\frac12} d_{j}^{-1+\frac{d}{2}} 
\sum_{*,i}(h^{1+\alpha}\vertiii{F_t}_{Q_i} 
+ h^{\alpha} \vertiii{F}_{1,Q_i})
d_i^{1-\alpha} \bigg(\frac{\min (d_i ,d_j )}{\max (d_i ,d_j )}\bigg)^{\alpha} . 
\end{align*}

In the case `$d=2$ (in this case $\alpha>\frac12$)'' or ``$d=3$, $r\ge 2$ and $\Omega$ is convex (in this case $\alpha>1$)'', we have 
\begin{align*}
{\mathscr K}  
&\leq
C+CC_*^{3+d/2} +C 
\qquad\mbox{(here we exchange the order of summation)} \\
&\quad 
+C\sum_{*,i}(h^{1+\alpha}\vertiii{F_t}_{Q_i} 
+ h^{\alpha} \vertiii{F}_{1,Q_i})d_i^{1-\alpha}
\sum_j  \bigg(\frac{h}{d_j}\bigg)^{-\frac12} d_{j}^{-1+\frac{d}{2}} 
\bigg(\frac{\min (d_i ,d_j )}{\max (d_i ,d_j )}\bigg)^{\alpha}\\
&\leq
C+CC_*^{3+\frac{d}{2}}    
+C\sum_{*,i}(h^{1+\alpha}\vertiii{F_t}_{Q_i} 
+ h^{\alpha} \vertiii{F}_{1,Q_i})d_i^{\frac{d}{2}-\alpha}  \bigg(\frac{h}{d_i}\bigg)^{-\frac12} 
\quad\mbox{(\eqref{lalnu} is used)}\\
&=
C+CC_*^{3+\frac{d}{2}}    
+C\sum_{*,i}d_i^{1+\frac{d}{2}}\left(\vertiii{F_t}_{Q_i} 
\bigg(\frac{h}{d_i}\bigg)^{\frac12+\alpha} 
+ d_i^{-1}\vertiii{F}_{1,Q_i}\bigg(\frac{h}{d_i}\bigg)^{\alpha-\frac12}\right) \\
&\leq 
C+CC_*^{3+\frac{d}{2}}    
+Cd_{*}^{1+\frac{d}{2}}
\left(\vertiii{F_t}_{Q_*} + d_j^{-1}\vertiii{F}_{1,Q_*} \right) \\
&\quad
+C\sum_{i}\bigg(\frac{h}{d_i}\bigg)^{-\frac12}  d_i^{1+\frac{d}{2}}\left(\vertiii{F_t}_{Q_i} 
+ d_j^{-1}\vertiii{F}_{1,Q_i}\right)\bigg(\frac{h}{d_i}\bigg)^{\alpha} \\
&\leq  C+C C^{3+d/2 }_* +\frac{C {\mathscr K} }{C_*^{\alpha}}  .
\end{align*}
By choosing $C_*$ to be large enough ($C_*$ is determined now), the term $\displaystyle\frac{C {\mathscr K} }{C_*^{\alpha}} $ will be absorbed by the left-hand side of the inequality above. In this case, the inequality above implies 
\begin{align}
{\mathscr K} \le C .
\end{align}  
Substituting the last inequality into \eqref{Bd31K2} yields 
\begin{align}\label{L1FtQ1}
h^{-1} \|\nabla F\|_{L^1((0,1)\times D_h)} 
&\leq C .
\end{align}

{\it Part II.}$\,\,\,$ 
For $t\in[1,\infty)$, we consider the error equation 
\begin{align}\label{Errror-Gamma-h}
\left\{\begin{array}{ll}
	(\partial_t (\Gamma_h(t) -\Gamma(t)) ,v_h)+(\nabla(\Gamma_h(t)-\Gamma(t)),\nabla v_h)=0 ,
	&\forall \, v_h\in \mathring S_h,\\[5pt]
	\Gamma_h(0) - \Gamma(0) =  \delta_{h,x_0} - \tilde\delta_{x_0}.
\end{array}\right.
\end{align}
Let $\chi(t)$ be a smooth cut-off function such that
$$
\chi(t)= \left\{
\begin{aligned}
&1 &&\mbox{for}\,\,\, t\ge \frac12 ,\\
&0 &&\mbox{for}\,\,\, t\le \frac14 .
\end{aligned}
\right.
$$
Then the function $\eta(t)= \chi(t) (\Gamma_h(t) -\Gamma(t))$ satisfies the following equation: 
\begin{align}\label{Errror-Gamma-h-2}
\left\{\begin{array}{ll}
	(\partial_t \eta(t),v_h)+(\nabla \eta(t) , \nabla v_h)
	= ( \partial_t \chi \, (\Gamma_h(t) -\Gamma(t)) ,v_h)  ,
	&\forall \, v_h\in \mathring S_h,\\[5pt] 
	\eta(0) =  0 ,
\end{array}\right.
\end{align}
which can also be written as 
\begin{align}\label{Errror-Gamma-h-3}
\left\{\begin{array}{ll}
	\partial_t R_h\eta - \Delta_h R_h\eta  
	= - \partial_t [\chi (P_h-R_h)\Gamma]   + \partial_t \chi(t) (\Gamma_h(t) - P_h\Gamma(t)) , \\[5pt] 
	R_h\eta(0) =  0 , 
\end{array}\right.
\end{align}
where $-\Delta_h: \mathring S_h\to \mathring S_h$ is the discrete Laplacian defined by
\begin{equation}\label{eq: discrete Laplace}
(-\Delta_hv_h,\chi)=(\nabla v_h,\nabla \chi), \quad \forall \, v_h\in \mathring S_h.
\end{equation}
The solution to \eqref{Errror-Gamma-h-3} can be furthermore expressed in terms of the semigroup $e^{t\Delta_h}$, i.e., 
\begin{align}
R_h\eta(t)  
&= - \int_0^t e^{(t-s)\Delta_h} \partial_s [\chi(s) (P_h-R_h)\Gamma(s)] \d s \notag\\
&\quad+ \int_0^t e^{(t-s)\Delta_h} \partial_s \chi(s) (\Gamma_h(s) - P_h\Gamma(s)) \d s . 
\end{align}
By applying the operator $(-\Delta_h)^{\frac12} $ to the both sides of the above equality, we obtain 
\begin{align}
(-\Delta_h)^{\frac12}  R_h\eta(t)  
= 
&\, - \int_0^t (-\Delta_h)^{\frac12} e^{(t-s)\Delta_h} \partial_s [\chi(s) (P_h-R_h)\Gamma(s)] \d s \notag\\
&\, + \int_0^t (-\Delta_h)^{\frac32} e^{(t-s)\Delta_h} (-\Delta_h)^{-1} \partial_s \chi(s) (\Gamma_h(s) - P_h\Gamma(s)) \d s ,
\end{align}
and therefore 
\begin{align}\label{RQ4}
\| \nabla R_h\eta(t) \|_{L^2} 
&=  
\| (-\Delta_h)^{\frac12} R_h\eta(t) \|_{L^2} \notag\\
&\le  
\bigg\| \int_0^t  e^{(t-s)\Delta_h} (-\Delta_h)^{\frac12} \partial_s [\chi(s) (P_h-R_h)\Gamma(s)] \d s\bigg\|_{L^2} \notag\\
&\quad\,
+\bigg\|\int_{0}^t (-\Delta_h)^{\frac32}  e^{(t-s)\Delta_h}(-\Delta_h)^{-1} \partial_s \chi(s) (\Gamma_h(s) - P_h\Gamma(s)) \d s \bigg\|_{L^2} \notag\\
&\le 
C \int_{\frac{1}{4}}^t e^{-\lambda_0 (t - s)} \| \nabla (P_h-R_h)\partial_s [\chi(s)\Gamma(s)]\|_{L^2} \d s \notag\\
&\quad\, + C \int_{\frac14}^{\frac12} (t-s)^{-\frac32}  \| (-\Delta_h)^{-1} (\Gamma_h(s) - P_h\Gamma(s))\|_{L^2} \d s ,
\end{align}
where $\lambda_0 > 0$ is the smallest eigenvalue of $-\Delta$ and we have used the following exponential decay estimate for the discrete semigroup in the last inequality:
\begin{align}
	\| e^{t\Delta_h} u_h \|_{L^2} \leq C e^{-\lambda_0 t} \| u_h \|_{L^2} , 
\end{align}
for any $t \geq 0$ and $u_h\in \mathring S_h(\Omega)$. 
{In addition, we have used the following result in \eqref{RQ4}: 
$$
\|(-\Delta_h)^{3/2}e^{(t-s)\Delta_h}\|_{L^2(\Omega)\to L^2(\Omega)}\le C(t-s)^{-3/2} , 
$$
which follows from the smoothing estimates for the analytic semigroup $e^{(t-s)\Delta_h}$, which means that (see \cite[Theorem 2.1 and Remark 2.1]{Li19}) 
$$
\|\partial_t^l e^{t\Delta_h}\chi\|_{L^2(\Omega)}\le Ct^{-l} \|\chi\|_{L^2(\Omega)}
\,\,\,\mbox{for}\,\,\,\forall\chi\in \mathring S_h \,\,\,\mbox{and} \,\,\, l=0,1,2,...  
$$ 
Since $\partial_t^l e^{t\Delta_h}\chi = \Delta_h^l e^{t\Delta_h}\chi $ for $\chi\in \mathring S_h$, it follows that 
$$
\|(-\Delta_h)^le^{t\Delta_h}\chi\|_{L^2(\Omega)}\le Ct^{-l}\|\chi\|_{L^2(\Omega)}
\,\,\,\mbox{for}\,\,\,\forall\chi\in \mathring S_h \,\,\,\mbox{and} \,\,\, l=1,2 , 
$$ 
which imply that
$$
\begin{aligned}
\|(-\Delta_h)^{3/2}e^{(t-s)\Delta_h}\chi\|^2_{L^2(\Omega)}&=\left((-\Delta_h)^{3/2}e^{(t-s)\Delta_h}\chi,(-\Delta_h)^{3/2}e^{(t-s)\Delta_h}\chi\right)\\
&=\left((-\Delta_h)^{2}e^{(t-s)\Delta_h}\chi,(-\Delta_h)e^{(t-s)\Delta_h}\chi\right)\\
&\le \|\Delta_h^{2}e^{(t-s)\Delta_h}\chi\|_{L^2(\Omega)}\|\Delta_he^{(t-s)\Delta_h}\chi\|_{L^2(\Omega)}\\
&\le \frac{C}{(t-s)^3}\|\chi\|^2_{L^2(\Omega)} . 
\end{aligned}
$$
This proves the desired result which we use in \eqref{RQ4}.}

By integrating the above inequality for $t\in[1,+\infty)$ and using Fubini's theorem and \eqref{eq:Ph_app}-\eqref{eq:Rh_app}, we obtain 
\begin{align}\label{eq:eta_grad_L2}
&\| \nabla R_h\eta \|_{L^1(1,\infty;L^2)} \notag\\
&\le 
C\| \nabla (P_h-R_h)\partial_t (\chi \Gamma) \|_{L^1(\frac14,+\infty;L^2)} 
+C \| (-\Delta_h)^{-1}  (\Gamma_h-P_h\Gamma) \|_{L^1(0,\frac12;L^2)} \notag\\
&\le 
Ch^{\alpha}\| \partial_t (\chi \Gamma) \|_{L^1(\frac14,+\infty;H^{1+\alpha})} 
+C \| \Gamma_h-\Gamma \|_{L^1(\frac14,\frac12;L^1)} \notag\\
&\le 
Ch^{\frac12} ,
\end{align}
where we have used Lemma \ref{lemma:three} in deriving the second and the last inequality.

From \eqref{eq:eta_grad_L2} and Poincar\'e inequality,
$$
\|\Gamma_h - R_h\Gamma\|_{L^1(1,+\infty;H^1(\Omega))} \le Ch^{\frac12}. 
$$
By using the triangle inequality, H\"older's inequality on $D_h$, \eqref{eq:Rh_app} and Lemma \ref{lemma:three} (item 1), we obtain
\begin{align}\label{W11-Gamma-2}
&\|\Gamma_h - \Gamma\|_{L^1(1,\infty;W^{1,1}(D_h))} \notag\\
&\le
Ch^{\frac12}\|\Gamma_h - R_h\Gamma\|_{L^1(1,\infty;H^1(D_h))}
+Ch^{\frac12} \|R_h\Gamma-\Gamma\|_{L^1(1,\infty;H^{1}(D_h))} \notag \\
&\le
Ch 
+Ch^{\frac12+\alpha} \| \Gamma \|_{L^1(1,\infty;H^{1+\alpha}(\Omega))} \notag\\
&\le Ch. 
\end{align}
Estimates \eqref{L1FtQ1}--\eqref{W11-Gamma-2} imply \eqref{grad-Fh-estimate}. 
This completes the proof of Theorem \ref{thm:semi-WMP}. 
\qed

\section{Weak maximum principle of fully discrete FEM}\label{sec: fully discrete}
\setcounter{equation}{0}

To prove Theorem \ref{THM:BDF-FEM}, it suffices to show
\begin{align}
\|u_h^N\|_{L^\infty(\Omega)}
\le 
C \max_{1\le n\le k-1} \|u_h^n\|_{L^\infty(\Omega)}
+ C\max_{k\le n\le N}\|g_h^n\|_{L^\infty(\partial\Omega)}
\quad\forall\, N\ge k , 
\end{align}
with a constant $C$ independent of $\tau$, $h$ and $N$. 
To this end, we decompose the solution into three parts, i.e.,
\begin{align}\label{uhN-decomp}
u_h^N = u_{h,1}^N + u_{h,2}^N + u_{h,3}^N ,
\end{align}
where $u_{h,1}^n$, $u_{h,2}^n$ and $u_{h,3}^n$ are finite element solutions of the following problems: 
\begin{align}
\label{Eq-uh1}
&\left\{
\begin{aligned}
&\bigg(\frac{1}{\tau} \sum_{j=0}^k \delta_j u_{h,1}^{n-j},v_h\bigg) + (\nabla u_{h,1}^n,\nabla v_h)=0 &&\forall\, v_h\in \mathring S_h, &&n=k,\dots,N,\\[1pt]
&u_{h,1}^n=g_h^n(\delta_{n,N-1} + \delta_{n,N}) &&\mbox{on}\,\,\,\partial\Omega , &&n=k,\dots,N,\\[1pt]
&u_{h,1}^j=0 &&\mbox{in}\,\,\,\Omega, &&j=0,\dots,k-1,\\[1pt]
\end{aligned}
\right.	\\[8pt]
\label{Eq-uh2}
&\left\{
\begin{aligned}
&\bigg( \frac{1}{\tau} \sum_{j=0}^k \delta_j u_{h,2}^{n-j},v_h\bigg) + (\nabla u_{h,2}^n,\nabla v_h)=0 &&\forall\, v_h\in \mathring S_h, &&n=k,\dots,N,\\[1pt]
&u_{h,2}^n=g_h^n(1 - \delta_{n,N-1} - \delta_{n,N}) &&\mbox{on}\,\,\,\partial\Omega , &&n=k,\dots,N,\\[1pt]
&u_{h,2}^j=0 &&\mbox{in}\,\,\,\Omega, &&j=0,\dots,k-1,\\[1pt]
\end{aligned}
\right.\\[8pt]
\label{Eq-uh3}
&\left\{
\begin{aligned}
&\bigg( \frac{1}{\tau} \sum_{j=0}^k \delta_j u_{h,3}^{n-j},v_h\bigg) + (\nabla u_{h,3}^n,\nabla v_h)=0 &&\forall\, v_h\in \mathring S_h, &&n=k,\dots,N,\\[1pt]
&u_{h,3}^n=0 &&\mbox{on}\,\,\,\partial\Omega, &&n=k,\dots,N,\\[1pt]
&u_{h,3}^j=u_{h}^j &&\mbox{in}\,\,\,\Omega, &&j=0,\dots,k-1,\\[1pt]
\end{aligned}
\right.	
\end{align}
with $\delta_{n,N}$ denoting the Kronecker symbol. In the following three subsections, we prove the following three estimates:
\begin{align}
&\label{uN-Linft-proved}
\|u_{h,1}^N\|_{L^\infty(\varOmega)} 
\begin{cases}
	\leq C \max\limits_{n = N} \|g_{h}^n \|_{L^\infty(\partial\varOmega)} , \quad & N=k ,  \\
	\leq C \max\limits_{n = N-1, N} \|g_{h}^n \|_{L^\infty(\partial\varOmega)},  \quad & N\geq k+1 ,
\end{cases}     
\\
&\label{BDF-FEM-f}
\|u_{h,2}^N\|_{L^\infty(\varOmega)}
\begin{cases}
	= 0 , \quad & N=k, k+1 ,  \\
	\leq C \max\limits_{k \leq n \leq N-2} \|g_{h}^n \|_{L^\infty(\partial\varOmega)},  \quad & N\geq k+2 ,
\end{cases}     
\\
&\label{Eq-uh3-Linfty}
\|u_{h,3}^N\|_{L^\infty(\varOmega)}
\le
C \max_{1\le n\le k-1} \| u_h^n \|_{L^\infty(\varOmega)} .
\end{align}
Then, substituting these estimates into \eqref{uhN-decomp}, we obtain the desired result in Theorem \ref{THM:BDF-FEM}.

\subsection{Proof of \eqref{uN-Linft-proved}}

We define the sector $\Sigma_{\theta} = \{z\in\C : |\arg (z)| \leq \theta \}$. 
The following result will be used in both this and next subsections. 
\begin{lemma}\label{Lemma:z-bd}
For $\theta\in(-\frac{\pi}{2},\frac{\pi}{2})$ and $z\in \Sigma_{\theta}$, the elliptic finite element problem 
\begin{align}
&\left\{
\begin{aligned}
&( z \phi_h , v_h ) + (\nabla \phi_h,\nabla v_h)=0 &&\forall\, v_h\in \mathring S_h, \\[1pt]
&\phi_h=f_h &&\mbox{on}\,\,\,\partial\Omega , \notag
\end{aligned}
\right.	
\end{align}
is well defined for any $f_h\in S_h(\partial\Omega)$ and satisfies the following estimate: 
\begin{align}
\|\phi_h\|_{L^\infty(\varOmega)}
\le
C\|f_h\|_{L^\infty(\partial\varOmega)} , \notag
\end{align}
where the constant $C$ is independent of $z\in\Sigma_\theta$. 

\end{lemma}
\begin{proof}
Let $w_h\in S_h$ be the finite element solution of the elliptic problem 
\begin{align}\label{elliptic-wh}
&\left\{
\begin{aligned}
&(\nabla w_h,\nabla v_h)=0 &&\forall\, v_h\in \mathring S_h, \\[1pt]
&w_h=f_h &&\mbox{on}\,\,\,\partial\Omega .
\end{aligned}
\right.	
\end{align}
Then $\phi_h-w_h\in\mathring S_h$ is the finite element solution of the zero Dirichlet boundary value problem
\begin{equation*} 
(z(\phi_h-w_h), v_h)+(\nabla(\phi_h-w_h),\nabla  v_h) = -(zw_h , v_h)\quad \forall\,  v_h\in\mathring S_h ,
\end{equation*}
which can equivalently be written as
\begin{equation*} 
(z-\varDelta_h)(\phi_h-w_h) = -z P_hw_h ,
\end{equation*}
where $P_h$ is the $L^2$ projection operator from $L^2(\Omega)$ onto its closed subspace $\mathring S_h$. (Since $w_h\neq 0$ on $\partial\varOmega$, it follows that $P_hw_h\neq w_h$.)

For the elliptic problem \eqref{elliptic-wh}, it is known that the following weak maximum principle holds (cf.\ \cite{Schatz80} and \cite{Leykekhman_Li_2021} for 2D and 3D cases, respectively):
\begin{equation}\label{weak-max-elliptic}
\|w_h\|_{L^\infty(\varOmega)} \le 
C\|f_h\|_{L^\infty(\partial\varOmega)} .
\end{equation}
Hence, we have 
\begin{align}\label{uN-wN-Linfty}
\|\phi_h-w_h\|_{L^\infty(\varOmega)} 
= \|z(z-\varDelta_h)^{-1} P_h w_{h} \|_{L^\infty(\varOmega)} 
\le 
C\|w_{h} \|_{L^\infty(\varOmega)} ,
\end{align}
where the last inequality is due to the $L^\infty$ stability of $P_h$ (cf. \cite[Lemma 6.1]{Tho06}) and the resolvent estimate (cf. \cite[Theorem 3.1]{Li22})
$$
\|z(z-\varDelta_h)^{-1} \|_{L^\infty(\varOmega)\rightarrow L^\infty(\varOmega)}
\le
C\quad\forall\, z\in\C,\,\,\, z\in\Sigma_\theta . 
$$
By using the triangle inequality and \eqref{weak-max-elliptic}--\eqref{uN-wN-Linfty}, we obtain \eqref{uN-Linft-proved}. 
\end{proof}

For $N= k$, according to the initial condition, $u_{h,1}^n$ vanishes for $1\le n\le k-1$, and $u_{h,1}^{N}$ is the solution of the discrete elliptic equation
\begin{align}\label{Eq-uh1-N2}
	&\left\{
	\begin{aligned}
		&\bigg( \frac{u_{h,1}^{N}}{\tau},v_h\bigg) + (\nabla u_{h,1}^N,\nabla v_h)=0 &&\forall\, v_h\in \mathring S_h, \\[1pt]
		&u_{h,1}^{N}=g_h^{N} &&\mbox{on}\,\,\,\partial\Omega ,
	\end{aligned}
	\right.	
\end{align}
where we have used the fact that $\delta_0 = 1$ for all BDF-$k$ methods with $k=1,...,6$.
Then Lemma \ref{Lemma:z-bd} implies that the solution of \eqref{Eq-uh1-N2} satisfies \eqref{uN-Linft-proved}. 

If $N\geq k+1$, the solution $u_{h,1}^n$ of \eqref{Eq-uh1} is actually zero for $1\le n\le N-2$. Therefore $u_{h,1}^{N-1}$ satisfies the elliptic equation
\begin{align}\label{Eq-uh1-N-1}
&\left\{
\begin{aligned}
&\bigg( \frac{u_{h,1}^{N-1}}{\tau},v_h\bigg) + (\nabla u_{h,1}^{N-1},\nabla v_h)=0 &&\forall\, v_h\in \mathring S_h, \\[1pt]
&u_{h,1}^{N-1}=g_h^{N-1} &&\mbox{on}\,\,\,\partial\Omega ,
\end{aligned}
\right.	
\end{align}
and similarly $u_{h,1}^N$ is determined by
\begin{align}\label{Eq-uh1-N}
	&\left\{
	\begin{aligned}
		&\bigg( \frac{u_{h,1}^N}{\tau},v_h\bigg) + (\nabla u_{h,1}^N,\nabla v_h)
		=
		- \bigg( \frac{ \delta_1 u_{h,1}^{N-1}}{\tau},v_h\bigg)
		 &&\forall\, v_h\in \mathring S_h, \\[1pt]
		&u_{h,1}^N=g_h^N &&\mbox{on}\,\,\,\partial\Omega .
	\end{aligned}
	\right.	
\end{align}
Since $\frac{1}{\tau}\in \Sigma_\theta$, Lemma \ref{Lemma:z-bd} implies that the solution of \eqref{Eq-uh1-N-1} satisfies 
\begin{align}\label{eq:uh1N-1}
	\| u_{h,1}^{N-1} \|_{L^\infty(\Omega)} \leq C \| g_{h}^{N-1} \|_{L^\infty(\partial\Omega)} .
\end{align}
At the time level $N$, we decompose $u_{h,1}^{N}=u_{h,1}^{N, {\rm homo}}+u_{h,1}^{N, {\rm inhomo}}$ into  homogeneous and inhomogeneous parts. By applying Lemma \ref{Lemma:z-bd} to homogeneous part of \eqref{Eq-uh1-N} to get
\begin{align}
	\| u_{h,1}^{N, {\rm homo}} \|_{L^\infty(\Omega)} \leq C \| g_{h}^{N}  \|_{L^\infty(\partial\Omega)} , \notag
\end{align}
and for the inhomogeneous part we have,
\begin{align}
	\| u_{h,1}^{N, {\rm inhomo}} \|_{L^\infty(\Omega)}
	=
	\| (\tau^{-1} - \Delta_h)^{-1} P_h \tau^{-1} \delta_1 u_{h,1}^{N-1} \|_{L^\infty(\Omega)} 
	&\leq C \| u_{h,1}^{N-1} \|_{L^\infty(\Omega)} \notag\\
	&\leq C \| g_{h}^{N-1} \|_{L^\infty(\partial\Omega)} , \notag
\end{align}
where we have used the resolvent estimate and the $L^\infty$ stability of $P_h$ in the first inequality, and have used \eqref{eq:uh1N-1} in the second inequality. Then by the triangle inequality, we conclude \eqref{uN-Linft-proved} as follows
\begin{align}
\| u_{h,1}^{N} \|_{L^\infty(\Omega)} 
&\leq \| u_{h,1}^{N, {\rm homo}} \|_{L^\infty(\Omega)} + \| u_{h,1}^{N, {\rm inhomo}} \|_{L^\infty(\Omega)} \notag\\
&\leq C (\| g_h^{N-1} \|_{L^\infty(\partial\Omega)} + \| g_h^N \|_{L^\infty(\partial\Omega)}) . \notag
\end{align}

\subsection{Proof of \eqref{BDF-FEM-f}}

According to \eqref{Eq-uh2}, $u_{h,2}^n, k\leq n\leq N$ is trivial for $N\leq k+1$. So in this subsection we will assume $N \geq k+2$ in order to get at least one meaningful time-stepping.

We extend $g_{h}^n$ to be zero for $0\leq n\leq k-1$ and $n\ge N+1$, and redefine $g_h^{N-1}=0$ and $g_h^N=0$. Moreover we extend $u_{h,2}^n, n\geq N+1$, to be the solution of the following equation at the time level $n\geq N+1$:

\begin{align}\label{Eq-uh2-2}
	&\left\{
	\begin{aligned}
		&\bigg( \frac{1}{\tau} \sum_{j=0}^k \delta_j u_{h,2}^{n-j},v_h\bigg) + (\nabla u_{h,2}^n,\nabla v_h)=0 &&\forall\, v_h\in \mathring S_h, &&n\geq k,\\[1pt]
		&u_{h,2}^n=g_h^n  &&\mbox{on}\,\,\,\partial\Omega , &&n\geq k,\\[1pt]
		&u_{h,2}^j=0 &&\mbox{in}\,\,\,\Omega, &&j=0,\dots,k-1 . \\[1pt]
	\end{aligned}
	\right.
\end{align} 
We denote by
$$
\tilde u_{h,2}(\zeta) = \sum_{n=k}^{\infty} u_{h,2}^n \zeta^n
\quad\mbox{and}\quad
\tilde g_h(\zeta)= \sum_{n=k}^{\infty} g_h^n \zeta^n
$$ 
the generating functions of the two sequences $(u_{h,2}^n)_{n=k}^{\infty}$ and $(g_h^n)_{n=k}^{\infty}$, respectively. By definition, $g_h^n = 0$ for $n\geq N$. From the standard energy decay estimate we know the boundedness $\sup_{n\geq 0}\| u_{h,2}^n \|_{L^2(\Omega)} \leq C \sup_{n\geq N}\| u_{h,2}^n \|_{L^2(\Omega)} \leq C_{h, N}$, and moreover via the inverse inequality and the H\"older's inequality, it holds that $\sup_{n\geq 0}\| u_{h,2}^n \|_{L^p(\Omega)} \leq C_{h, N}$ for all $p\in [1,\infty]$. Therefore the $L^p$-valued generating functions $\tilde g_h(\zeta)$ and $\tilde u_{h,2}(\zeta)$ are analytic in $\zeta\in\C$ and in the open unit disk $\zeta\in \D = \{z\in \C : |z|<1 \}$ respectively for any $h > 0$ and $p\in [1,\infty]$.

For any $\zeta\in \D$, we observe the following relation of generating functions
\begin{align}\label{eq:gen_iden}
	\delta(\zeta) \tilde u_{h,2}(\zeta)
	&= \Big(\sum_{j=0}^k \delta_j \zeta^j \Big) \Big(\sum_{j=k}^{+\infty} u_{h,2}^j \zeta^j \Big) \notag\\
	&= \sum_{n=k}^{+\infty} \zeta^n \sum_{j=0, n-j \geq k}^k \delta_j u_{h,2}^{n-j} \notag\\
	&= \sum_{n=k}^{+\infty} \zeta^n \sum_{j=0}^k \delta_j u_{h,2}^{n-j}
	- \sum_{n=k}^{2k-1} \zeta^n \sum_{j=0, n-j < k}^k \delta_j u_{h,2}^{n-j} .
\end{align}
Multiplying \eqref{Euler-FEM} by $\zeta^n$, summing up the results for $n=k,k+1,\dots$ and using the identity \eqref{eq:gen_iden} and the initial condition $u_{h,2}^n = 0 $ for $n = 0,...,k-1$, we obtain the following discrete elliptic equation of the generating function $\tilde u_{h,2}(\zeta)$: 
\begin{align}\label{BDF-FE-Laplce}
\left\{
\begin{aligned}
&\bigg(\delta_\tau(\zeta) \tilde u_{h,2}(\zeta) ,v_h \bigg) + (\nabla \tilde u_{h,2}(\zeta),\nabla v_h)
= \bigg(- \frac{1}{\tau} \sum_{n=k}^{2k-1} \zeta^n \sum_{j=0, n-j < k}^k \delta_j u_{h,2}^{n-j}, v_h \bigg) = 0 \\[3pt] 
&\hspace{310pt}\forall\, v_h\in \mathring S_h \\
&\tilde u_{h,2}(\zeta) = \tilde g_h(\zeta) \quad\mbox{on}\,\,\,\partial\Omega , 
\end{aligned}
\right.	
\end{align}
whose solution map is denoted by $\tilde M_h(\zeta):S_h(\partial\Omega)\rightarrow S_h(\Omega), \tilde g_h(\zeta)\mapsto\tilde u_{h,2}(\zeta)$.

We define the truncated sectors $\Sigma_\theta^\tau = \{z\in \Sigma_{\theta}: {\rm Im}(z) \leq \pi/\tau \}$ and $\Sigma_{\theta,\sigma}^\tau = \{z\in \Sigma_{\theta}: |z|\geq \sigma^{-1} \mbox{ and }{\rm Im}(z) \leq \pi/\tau \}$ for any $\sigma > \tau / \pi$. The following result of the generating function of BDF-$k$ method is shown in \cite[Lemma 3.1]{JZ23}.
\begin{lemma}\label{lemma:BDF_sec}
	For any $\epsilon$ and $k=1,...,6$, there exists $\theta_\epsilon\in(\pi/2,\pi)$ such that for any  $\theta\in(\pi/2,\theta_\epsilon)$, there exist positive constants $C$, $C_1$,and $C_2$ (independent of $\tau$) such that for all $z\in\Sigma_{\theta}^\tau$ we have the inclusion $\delta_\tau(e^{-\tau z})\subseteq \Sigma_{\pi-\theta_k+\epsilon}$ (Recall that $\theta_k$ is the angle of $A$-stability region.) and the following estimates hold
	\begin{align}
		C_1 |z| \leq |\delta_\tau(&e^{-\tau z})| \leq C_2 |z| , \notag\\
		\delta_\tau(e^{-z\tau}) &\in \Sigma_{\pi - \theta_k +\epsilon} , \notag\\
		| \delta_\tau(e^{-z\tau}) - z | &\leq C \tau^k |z|^{k+1} . \notag
	\end{align}
\end{lemma}
For BDF-1 (i.e. backward Euler) time-stepping method, the results in Lemma \ref{lemma:BDF_sec} can be furthermore improved.
\begin{lemma}\label{lemma:Euler_sec}
	Let $\theta\in (\frac{\pi}{2}, {\rm arccot(-\frac{2}{\pi})})$. For $z\in \Sigma_{\theta}$, we have $\frac{1 - e^{-\tau z}}{\tau}\in \Sigma_{\theta}$. Moreover there exist positive constants $C$, $C_1$,and $C_2$ such that for $z\in \Sigma_{\theta}^\tau$
	\begin{align}
		&C_1 |z| \leq |\frac{1 - e^{-\tau z}}{\tau}| \leq C_2 |z| , \notag\\
		&| \frac{1 - e^{-\tau z}}{\tau} - z | \leq C \tau |z|^{2} . \notag
	\end{align}
	Additionally, for $\sigma>\tau/\pi$, the following inequalities hold:  
	\begin{align}
		&\bigg| 1-\frac{e^{\tau z}-1}{z\tau}	\bigg|
		\le
		C|z|\tau 
		\quad\mbox{and}\quad 
		\bigg| \frac{e^{\tau z}-1}{z\tau}	\bigg|
		\leq C
		&&\forall\,  z\in \Gamma_{\theta,\sigma}^\tau , 
		\label{exp-tau-taylor} \\
		&\bigg|\frac{e^{\tau z}-1}{z\tau}\bigg|\le C 
		\quad\mbox{and}\quad 
		|z|\le C \bigg| \frac{1-e^{-\tau z}}{\tau} \bigg| 
		&&\forall\,  z\in \Gamma_{\theta,\sigma}\backslash\Gamma_{\theta,\sigma}^\tau . 
		\label{exp-tau-z}
	\end{align}
\end{lemma}
The first part of this lemma results from \cite[Lemma 3.4]{GLW19} and \cite[Appendix C]{GLW19b}. For the second part, inequalities in \eqref{exp-tau-taylor} can be verified by using Taylor's expansion on a bounded domain. Inequalities in \eqref{exp-tau-z} can be proved by using the fact that $|z|\tau\ge \pi/ | \sin(\theta) |$ for $z\in \Gamma_{\theta,\sigma}\backslash\Gamma_{\theta,\sigma}^\tau$. 

From now on, we fix any $0 < \epsilon < \min_{1\leq k\leq 6} \theta_k$ and an angle $\theta$ such that 
\begin{align}\label{eq:theta}
	\frac{\pi}{2} <\theta < \min \{\theta_\epsilon,  {\rm arccot(-\frac{2}{\pi})} \}
\end{align}
where $\theta_\epsilon$ is determined in Lemma \ref{lemma:BDF_sec}.
From Lemma \ref{lemma:BDF_sec}, we know $\delta_\tau(e^{-z\tau}) \in \Sigma_{\pi - \theta_k +\epsilon} \subseteq \rho(\Delta_h)$ for $z\in \Sigma_{\theta}^\tau$, so the operator $\tilde M(e^{-\tau z})$ is well-defined and analytic for $z\in \Sigma_{\theta}^\tau$, satisfying the following estimate according to Lemma \ref{Lemma:z-bd}: 
\begin{align}\label{Mhz-Linfty}
\| \tilde M_h(e^{-\tau z}) \|_{L^\infty(\partial\varOmega)\rightarrow L^\infty(\varOmega)}
\le C 
\quad\forall\, z\in \Sigma_{\theta}^\tau .
\end{align}	
Since $\tilde u_{h,2}(\cdot)$ is analytic in the open unit disk $\D$, by Cauchy's integral formula, we derive for any $0 < \rho < 1$ that
\begin{align}\label{u_{h,2}n-repr}
u_{h,2}^N
&=\frac{1}{2\pi i}\int_{|\zeta|=\rho}
\tilde u_{h,2}(\zeta) \zeta^{-N-1} \d\zeta \notag \\
&=\frac{\tau}{2\pi i}\int_{-\frac{\ln\rho}{\tau} - i\frac{\pi}{\tau}}^{-\frac{\ln\rho}{\tau} + i\frac{\pi}{\tau}}
\tilde u_{h,2}(e^{-\tau z}) e^{t_Nz} \d z 
&&\mbox{(change of variable $\zeta=e^{-\tau z}$)} \notag \\
&=\frac{\tau}{2\pi i}\int_{\Gamma_{\theta,\sigma}^\tau}
\tilde u_{h,2}(e^{-\tau z}) e^{t_Nz}\d z 
&&\mbox{(deform the contour)} \notag \\
&=\frac{\tau}{2\pi i}\int_{\Gamma_{\theta,\sigma}^\tau}
\tilde M_h(e^{-\tau z})\tilde g_h(e^{-\tau z})e^{t_Nz}\d z 
&&\mbox{(by the definition of $M_h(\zeta)$)} ,
\end{align}
where we have deformed the contour of integration from 
$$\{z\in\C:{\rm Re}(z)= -\ln\rho/\tau,\,\, |{\rm Im}(z)| \le \pi/\tau \}$$ to 
\begin{align}
\begin{aligned}
\Gamma_{\theta,\sigma}^\tau
=\Gamma_{\theta,\sigma}^{\tau,1}\cup \Gamma_{\theta,\sigma}^{\tau,2}
\,\,\,\mbox{with}\,\,\,
\Gamma_{\theta,\sigma}^{\tau,1}
&=\{z\in\C: |{\rm arg}(z)|=\theta,\,\,|z|\ge s^{-1}\,\,\mbox{and}\,\,|{\rm Im}(z)|\le \pi/\tau\} ,\\
\,\,\mbox{and}\,\,\,
\Gamma_{\theta,\sigma}^{\tau,2}
&=\{z\in\C: |z|=s^{-1}\,\,\,\mbox{and}\,\,\, 
|{\rm arg}(z)|\le\theta \} , 
\end{aligned}
\end{align}
where $\sigma>\tau/\pi$ can be arbitrary. 
The deformation of the contour is valid due to the analyticity of the integrand and the periodicity in ${\rm Im}(z)$ with period $2\pi/\tau$. 

Let $u_{h,2}\in C([0,T]\times \Omega)\cap C^{\infty}_{\rm pw}([0,T];S_h(\Omega))$ be the globally continuous piecewise smooth solution of the semi-discrete FEM
\begin{align}\label{semi-discrete-FEM-2}
\left\{
\begin{aligned}
&(\partial_t u_{h,2}(t),v_h) + (\nabla u_{h,2}(t),\nabla v_h)=0 &&\forall\, v_h\in \mathring S_h,&&\forall\,t\in(0,T]\\ 
&u_{h,2}(t)=g_h(t) &&\mbox{on}\,\,\,\partial\Omega ,&&\forall\,t\in [0,T],\\
&u_{h,2}(0)=0 &&\mbox{in}\,\,\,\Omega ,
\end{aligned}
\right.	
\end{align}
where $g_h(t)=g_h^n$ is the $S_h(\partial\Omega)$-valued globally continuous piecewise linear function whose nodal values satisfy $g_h(t_n) = g_h^n$ for $n \geq 0$. In particular, $g_h(t)=0$ for $t\in [t_{N-1},t_N]$, which implies that 
\begin{align}
u_{h,2}(t) 
=e^{(t-t_{N-1})\varDelta_h}u_{h,2}(t_{N-1})
\quad\forall\, t\in [t_{N-1},t_N]. \notag
\end{align}
Hence, we can furthermore decompose $u_{h,2}(t)$ into 
\begin{align}
u_{h,2}(t)=u_{h,21}(t)+u_{h,22}(t) , \notag
\end{align}
with 
\begin{align*}
&u_{h,21}(t)
=\left\{
\begin{aligned}
&u_{h,2}(t) &&\mbox{for}\,\,\,t\in [0,t_{N-1}],\\
&0      &&\mbox{for}\,\,\,t\in(t_{N-1},t_N],
\end{aligned}
\right. \\[5pt]
&u_{h,22}(t)
=\left\{
\begin{aligned}
&0 &&\mbox{for}\,\,\,t\in [0,t_{N-1}],\\
&e^{(t-t_{N-1})\varDelta_h}u_{h,2}(t_{N-1})      &&\mbox{for}\,\,\,t\in(t_{N-1},t_N] ,
\end{aligned}
\right.
\end{align*}
and denote by 
$$
\hat v(z)
=\int_0^{+\infty} e^{-tz}v(t)\d t 
$$ 
the Laplace transform of a function $v$ in time. Then
\begin{align}
\hat u_{h,2}(z) 
=\hat u_{h,21}(z) + \hat u_{h,22}(z) 
=\hat u_{h,21}(z) + (z-\varDelta_h)^{-1}e^{-zt_{N-1}} u_{h,2}(t_{N-1}) . \notag
\end{align}
By considering the Lapalce transform of \eqref{semi-discrete-FEM-2} in time, we get
\begin{align}\label{semi-discrete-FEM-Laplce}
\left\{
\begin{aligned}
&(z\hat u_{h,2}(z),v_h) + (\nabla \hat u_{h,2}(z),\nabla v_h)=0 &&\forall\, v_h\in \mathring S_h, \\ 
&\hat u_{h,2}(z)=  \hat g_h(z)
&&\mbox{on}\,\,\,\partial\Omega .
\end{aligned} 
\right.	
\end{align}
We also have the following formula for the Laplace transform of the piecewise linear function $g_h(t)$.
\begin{lemma}\label{lemma:z-trans}
	\begin{align}\label{eq:z-trans}
		\hat g_h (z) = \frac{e^{-\tau z} + e^{\tau z} - 2}{z^2 \tau} \tilde g_h(e^{-\tau z}) .
	\end{align}
\end{lemma}
\begin{proof}
	We use the change of variables $\zeta=e^{-\tau z}$. Straightforward calculations together with the initial condition $g_h^0=0$ lead to
	\begin{align}
		\hat g_h (z) 
		&=
		\int_0^{+\infty} g_h(t) e^{-zt} \d t \notag\\
		&=
		\sum_{j=0}^{+\infty }\int_{t_j}^{t_{j+1}} \bigg( g_h^j + \frac{g_h^{j+1} - g_h^j}{\tau} (t - t_j) \bigg) e^{-zt} \d t \notag\\
		&=
		\sum_{j=0}^{+\infty } g_h^j \frac{\zeta^{j+1} - \zeta^{j}}{-z} 
		+ \sum_{j=0}^{+\infty } \frac{g_h^{j+1} - g_h^j}{\tau} \bigg( -\frac{\tau\zeta^{j+1}}{z} - \frac{\zeta^{j+1} - \zeta^{j}}{z^2} \bigg)
		\notag\\
		&=
		\sum_{j=0}^{+\infty } \frac{1 - \zeta}{z} g_h^j \zeta^j
		- \sum_{j=0}^{+\infty } \frac{1 - \zeta}{z} g_h^j \zeta^j
		- \sum_{j=0}^{+\infty } \frac{ (\zeta^{j+1} - \zeta^j) (g_h^{j+1} - g_h^j)}{z^2 \tau}
		\notag\\
		&=
		- \sum_{j=0}^{+\infty } \frac{ (1- \zeta) (1- \zeta^{-1}) }{z^2 \tau} g_h^j \zeta^j
		\notag\\
		&=
		\frac{e^{-\tau z} + e^{\tau z} - 2}{z^2 \tau} \tilde g_h(e^{-\tau z}) .
		\notag
	\end{align}
\end{proof}

Analogous to \eqref{exp-tau-taylor} and \eqref{exp-tau-z}, the estimates below follow from Taylor's expansion.
\begin{align}
	&\bigg| 1- \frac{e^{-\tau z} + e^{\tau z} - 2}{z^2 \tau^2} \bigg|
	\le
	C|z|^2\tau^2 
	\quad\mbox{and}\quad 
	\bigg| \frac{e^{-\tau z} + e^{\tau z} - 2}{z^2 \tau^2} \bigg|
	\leq C
	&&\forall\,  z\in \Gamma_{\theta,\sigma}^\tau , 
	\label{exp-tau-taylor1} \\
	&
	\Bigg|\frac{e^{-\tau z} + e^{\tau z} - 2}{z^2 \tau^2} \Bigg| \leq C |e^{-\tau z}|
	&&\forall\,  z\in \Gamma_{\theta,\sigma}\backslash\Gamma_{\theta,\sigma}^\tau . 
	\label{exp-tau-z1}
\end{align}
Then using the above lemma, we can rewrite \eqref{semi-discrete-FEM-Laplce} as  
\begin{align}\label{semi-discrete-FEM-Laplce-2}
\left\{
\begin{aligned}
&\bigg(\frac{1 - e^{-\tau z}}{\tau} \hat u_{h,2}(z),v_h\bigg) + (\nabla \hat u_{h,2}(z),\nabla v_h)
=
\bigg( \bigg(\frac{1 - e^{-\tau z}}{\tau} - z \bigg)\hat u_{h,2}(z),v_h \bigg)  \\
&\hspace{280pt} \forall\, v_h\in \mathring S_h, \\ 
&\hat u_{h,2}(z)=\ztr \tilde g_h(e^{-\tau z}) \quad \mbox{on}\,\,\,\partial\Omega .
\end{aligned} 
\right.	
\end{align}
We denote by $\tilde L_h(e^{-\tau z}):S_h(\partial\Omega)\rightarrow S_h(\Omega)$ the solution map of the homogeneous part of equation \eqref{semi-discrete-FEM-Laplce-2}. By the property of $\theta$ (equation \eqref{eq:theta}) and Lemma \ref{lemma:Euler_sec}, we have
\begin{align}\label{Lhz-Linfty}
	\| \tilde L_h(e^{-\tau z}) \|_{L^\infty(\partial\varOmega)\rightarrow L^\infty(\varOmega)}
	\le C 
	\quad\forall\, z\in \Sigma_{\theta} .
\end{align}	
The estimate of the operator difference $\tilde M_h(e^{-\tau z}) - \tilde L_h(e^{-\tau z})$ is given in the following lemma.
\begin{lemma}\label{lemma:M-L}
	Given any $z\in \Sigma_{\theta}^\tau$, it holds that
	\begin{align}
		\| \tilde M_h(e^{-\tau z}) - \tilde L_h(e^{-\tau z}) \|_{L^\infty(\partial\varOmega)\rightarrow L^\infty(\varOmega)}
		\leq C \tau |z| . \notag
	\end{align}
\end{lemma}
\begin{proof}
Again, we use the change of variables $\zeta=e^{-\tau z}$.	For any fixed $z\in \Sigma_\theta^\tau$ and any given boundary data $f_h\in S_h(\partial\Omega)$, if we denote $v_h(\zeta):=\tilde M_h(\zeta) f_h$ and $w_h(\zeta):=\tilde L_h(\zeta) f_h$, by the definitions of $\tilde M_h$ and $\tilde G_h$, we know
	\begin{align*}
		\left\{
		\begin{aligned}
			&\bigg(\delta_\tau(e^{-\tau z}) v_h,\phi_h \bigg) + (\nabla v_h,\nabla \phi_h)=0 &&\forall\, \phi_h\in \mathring S_h, \\ 
			& v_h = f_h
			&&\mbox{on}\,\,\,\partial\Omega ,
		\end{aligned} 
		\right.	
	\end{align*}
	and
	\begin{align*}
		\left\{
		\begin{aligned}
			&\bigg(\frac{1 - e^{-\tau z}}{\tau} w_h,\phi_h \bigg) + (\nabla w_h,\nabla \phi_h)=0 &&\forall\, \phi_h\in \mathring S_h, \\ 
			& w_h = f_h
			&&\mbox{on}\,\,\,\partial\Omega .
		\end{aligned} 
		\right.	
	\end{align*}
	Upon subtraction, the following equation holds for all $\phi_h\in \mathring S_h(\Omega)$
	\begin{align}
		\bigg(\frac{1 - e^{-\tau z}}{\tau} (v_h - w_h), \phi_h \bigg)
		+
		(\nabla (v_h - w_h), \nabla \phi_h)
		=
		-\bigg(\bigg(\delta_\tau(e^{-\tau z}) - \frac{1 - e^{-\tau z}}{\tau} \bigg) v_h, \phi_h \bigg) , \notag
	\end{align}
	and by inverting the elliptic operator
	\begin{align}
		v_h - w_h = - \bigg(\frac{1 - e^{-\tau z}}{\tau} - \Delta_h \bigg)^{-1} P_h \bigg(\bigg(\delta_\tau(e^{-\tau z}) - \frac{1 - e^{-\tau z}}{\tau}\bigg) v_h \bigg) .\notag
	\end{align} 
	Consequently, from the resolvent estimate, $L^\infty$ stability of $P_h$, Lemma \ref{lemma:BDF_sec}--\ref{lemma:Euler_sec} and the mapping property of $\tilde M_h(\zeta)$, we get
	\begin{align}
		\| (\tilde M_h(\zeta) - \tilde L_h(\zeta)) f_h \|_{L^\infty(\Omega)} 
		&=\| v_h - w_h \|_{L^\infty(\Omega)} \notag\\
		&\leq C |z|^{-1} \tau |z|^2 \| \tilde M_h(\zeta) f_h \|_{L^\infty(\Omega)} \notag\\
		&\leq C \tau |z| \| f_h \|_{L^\infty(\partial\Omega)} . \notag
	\end{align} 
	Therefore, it follows
	\begin{align}
		\| \tilde M_h(e^{-\tau z}) - \tilde L_h(e^{-\tau z}) \|_{L^\infty(\partial\varOmega)\rightarrow L^\infty(\varOmega)}
		\leq C \tau |z| . \notag
	\end{align} 
\end{proof}
The solution of \eqref{semi-discrete-FEM-Laplce-2} can be represented as
\begin{align}\label{hat-uh-z}
&\hat u_{h,2}(z) \notag\\
&=\tilde L_h(e^{-\tau z}) \ztr \tilde g_h(e^{-\tau z})
+\bigg(\frac{1 - e^{-\tau z}}{\tau} - \varDelta_h\bigg)^{-1} P_h \bigg(\frac{1 - e^{-\tau z}}{\tau} - z\bigg)\hat u_{h,2}(z) \notag \\
&=\tilde L_h(e^{-\tau z}) \ztr \tilde g_h(e^{-\tau z})
+\bigg(\frac{1 - e^{-\tau z}}{\tau} - \varDelta_h\bigg)^{-1} P_h \bigg(\frac{1 - e^{-\tau z}}{\tau} - z\bigg)\hat u_{h,21}(z) \notag \\
&\quad\,
+ \bigg(\frac{1 - e^{-\tau z}}{\tau} - \varDelta_h\bigg)^{-1} \bigg(\frac{1 - e^{-\tau z}}{\tau} - z\bigg) (z-\varDelta_h)^{-1}e^{-zt_{N-1}} u_{h,2}(t_{N-1}).
\end{align}
By using the inverse Laplace transform and the expression in \eqref{hat-uh-z}, we have
\begin{align}\label{u_{h,2}-tn-repr}
u_{h,2}(t_N)
=&\frac{1}{2\pi i}
\int_{\Gamma_{\theta,\sigma}} 
\hat u_{h,2}(z) e^{t_Nz}\d z \notag \\
=&\frac{\tau}{2\pi i}
\int_{\Gamma_{\theta,\sigma}^\tau} 
\tilde L_h(e^{-\tau z}) \ztrt \tilde g_h(e^{-\tau z}) e^{t_Nz}\d z \notag \\
&
+\frac{\tau}{2\pi i}
\int_{\Gamma_{\theta,\sigma}\backslash\Gamma_{\theta,\sigma}^\tau} 
\tilde L_h(e^{-\tau z}) \ztrt \tilde g_h(e^{-\tau z}) e^{t_Nz}\d z \notag \\
&+
\frac{1}{2\pi i}
\int_{\Gamma_{\theta,\sigma}} 
\bigg(\frac{1 - e^{-\tau z}}{\tau} - \varDelta_h\bigg)^{-1} P_h \bigg(\frac{1 - e^{-\tau z}}{\tau} - z\bigg)\hat u_{h,21}(z) e^{t_Nz}\d z \notag\\
&+
\frac{1}{2\pi i}
\int_{\Gamma_{\theta,\sigma}} 
\bigg(\frac{1 - e^{-\tau z}}{\tau} - \varDelta_h\bigg)^{-1} \bigg(\frac{1 - e^{-\tau z}}{\tau} - z\bigg) (z-\varDelta_h)^{-1} u_{h,2}(t_{N-1}) e^{\tau z}\d z ,  
\end{align}
where $\sigma>\tau/\pi$ can be arbitrary and different in each of the integral above. 

Subtracting \eqref{u_{h,2}-tn-repr} from \eqref{u_{h,2}n-repr} and using the boundary condition $g_h^n = 0$ for $n\leq k-1$ and $n \geq N - 1$, we obtain
\begin{align}\label{BDF-FEM-difference}
&u_{h,2}^N - u_{h,2}(t_N) \notag\\
&=
\frac{\tau}{2\pi i}\int_{\Gamma_{\theta,\sigma}^\tau}
\bigg(1- \ztrt \bigg) \tilde M_h(e^{-\tau z}) \tilde g_h(e^{-\tau z}) e^{t_Nz} \d z \notag\\
&\quad+
\frac{\tau}{2\pi i}\int_{\Gamma_{\theta,\sigma}^\tau}
\ztrt \bigg(\tilde M_h(e^{-\tau z}) -  \tilde L_h(e^{-\tau z}) \bigg) \tilde g_h(e^{-\tau z}) e^{t_Nz} \d z \notag\\
&\quad
- \frac{\tau}{2\pi i}
\int_{\Gamma_{\theta,\sigma}\backslash\Gamma_{\theta,\sigma}^\tau} 
\ztrt \tilde L_h(e^{-\tau z}) \tilde g_h(e^{-\tau z}) e^{t_Nz}\d z \notag\\
&\quad - \frac{1}{2\pi i}
\int_{\Gamma_{\theta,\sigma}} 
\bigg(\frac{1 - e^{-\tau z}}{\tau} - \varDelta_h\bigg)^{-1} \bigg(\frac{1 - e^{-\tau z}}{\tau} - z\bigg)\hat u_{h,21}(z) e^{t_Nz}\d z  \notag\\
&\quad - \frac{1}{2\pi i}
\int_{\Gamma_{\theta,\sigma}} 
\bigg(\frac{1 - e^{-\tau z}}{\tau} - \varDelta_h\bigg)^{-1} \bigg(\frac{1 - e^{-\tau z}}{\tau} - z\bigg) (z-\varDelta_h)^{-1} u_{h,2}(t_{N-1}) e^{\tau z}\d z \notag \\
&= 
\sum_{j=k}^{N-2} F^{N-j}_h g_h^j 
+\sum_{j=k}^{N-2} G^{n-j}_h g_h^j 
-\sum_{j=k}^{N-2} L^{n-j}_h g_h^j \notag\\
&\quad
-\int_{t_{k-1}}^{t_{N-1}}K_{h}(t_N-s) P_h u_{h,2}(s)\d s 
-Q_{h}u_{h,2}(t_{N-1}), 
\end{align}
where 
\begin{align} 
F^{n}_h
&= \frac{\tau}{2\pi i}\int_{\Gamma_{\theta,\sigma}^\tau}
\bigg(1-\ztrt \bigg) \tilde M_h(e^{-\tau z}) e^{t_nz} \d z \label{Fnh} , \\
G^n_h
&=
\frac{\tau}{2\pi i}\int_{\Gamma_{\theta,\sigma}^\tau}
\ztrt \bigg(\tilde M_h(e^{-\tau z}) -  \tilde L_h(e^{-\tau z})\bigg) e^{t_n z} \d z \label{Gnh} , \\
L^{n}_h
&=
\frac{\tau}{2\pi i}
\int_{\Gamma_{\theta,\sigma}\backslash\Gamma_{\theta,\sigma}^\tau} 
\ztrt \tilde L_h(e^{-\tau z}) e^{t_nz}\d z , \label{Lnh}\\
K_h(s)
&=
\frac{1}{2\pi i}
\int_{\Gamma_{\theta,\sigma}} 
\bigg(\frac{1 - e^{-\tau z}}{\tau} - \varDelta_h\bigg)^{-1} \bigg(\frac{1 - e^{-\tau z}}{\tau} - z\bigg) e^{s z}\d z  , \label{Knh} \\
Q_h 
&=
\frac{1}{2\pi i}
\int_{\Gamma_{\theta,\sigma}} 
\bigg(\frac{1 - e^{-\tau z}}{\tau} - \varDelta_h\bigg)^{-1} \bigg(\frac{1 - e^{-\tau z}}{\tau} - z\bigg) (z-\varDelta_h)^{-1}  e^{\tau z}\d z . \label{Qnh}
\end{align}

First, by choosing $\sigma=t_n$ in \eqref{Fnh} and using \eqref{Mhz-Linfty} and \eqref{exp-tau-taylor1}, we have 
\begin{align} 
\| F^{n}_h \|_{L^\infty(\partial\varOmega)\rightarrow L^\infty(\varOmega)}
&\le
C\tau \int_{\Gamma_{\theta,\sigma}^\tau} |z|^2 \tau^2 
\| \tilde M_h(e^{-\tau z}) \|_{L^\infty(\partial\varOmega)\rightarrow L^\infty(\varOmega)} e^{t_n{\rm Re}(z)} |\d z | 
\notag \\
&\le
C \tau^3  \int_{\Gamma_{\theta,t_n}^{\tau,1}} |z|^2
e^{t_n {\rm Re}(z)} |\d z |  
+
C \tau^3  \int_{\Gamma_{\theta,t_n}^{\tau,2}} |z|^2
 e^{t_n {\rm Re}(z)} |\d z | \notag \\ 
 &\le
 C \tau^3  \int_{t_n^{-1}}^{+\infty} r^2
 e^{-t_n r |\cos\theta|} \d r  
 +
 C \tau^3  \int_{-\theta}^\theta t_n^{-3}
 e^{-|\cos\theta|} \d\theta \notag \\ 
&\le
\frac{C\tau^3}{t_n^3} 
= \frac{C}{n^3} . \notag
\end{align}
Hence, 
\begin{align}\label{Fh-ghj}
\bigg\| \sum_{j=k}^{N-2} F^{N-j}_h g_h^j \bigg\|_{L^\infty(\varOmega)}
&\le
C\sum_{j=k}^{N-2}  \frac{C}{(N-j)^3}
\max_{k\le n\le N-2}\| g_h^n \|_{L^\infty(\partial\varOmega)} \notag\\
&\le
C\max_{k\le n\le N-2}\| g_h^n \|_{L^\infty(\partial\varOmega)} .
\end{align}

Analogously, from \eqref{exp-tau-taylor1}, Lemma \ref{lemma:M-L} and choosing $\sigma = t_n$, we get 
\begin{align} 
	\| G^{n}_h \|_{L^\infty(\partial\varOmega)\rightarrow L^\infty(\varOmega)}
	&\le
	C\tau \int_{\Gamma_{\theta,\sigma}^\tau} 
	\| \tilde M_h(e^{-\tau z}) - \tilde L_h(e^{-\tau z}) \|_{L^\infty(\partial\varOmega)\rightarrow L^\infty(\varOmega)} e^{t_n{\rm Re}(z)} |\d z | 
	\notag \\
	&\le
	C \tau^2  \int_{\Gamma_{\theta,t_n}^{\tau,1}} |z|
	e^{t_n {\rm Re}(z)} |\d z |  
	+
	C \tau^2  \int_{\Gamma_{\theta,t_n}^{\tau,2}} |z|
	e^{t_n {\rm Re}(z)} |\d z | \notag \\ 
	&\le
	C \tau^2 \int_{t_n^{-1}}^{+\infty} r
	e^{-t_n r |\cos\theta|} \d r  
	+
	C \tau^2  \int_{-\theta}^\theta
	t_n^{-2}
	e^{-|\cos\theta|} \d\theta \notag \\
	&\le
	\frac{C\tau^2}{t_n^{2}} 
	= \frac{C}{n^{2}} . \notag
\end{align}
Therefore, it follows 
\begin{align}\label{Gh-ghj}
	\bigg\| \sum_{j=k}^{N-2} G^{N-j}_h g_h^j \bigg\|_{L^\infty(\varOmega)}
	&\le
	C\sum_{j=k}^{N-2}  \frac{C}{(N-j)^{2}}
	\max_{k\le n\le N-2}\| g_h^n \|_{L^\infty(\partial\varOmega)} \notag\\
	&\le
	C\max_{k\le n\le N-2}\| g_h^n \|_{L^\infty(\partial\varOmega)} .
\end{align}

Then, by choosing $\sigma=t_{n-1}$ in \eqref{Lnh} and using \eqref{Lhz-Linfty} and the estimates in \eqref{exp-tau-z1}, we have 
\begin{align*}
\|L^{n}_h\|_{L^\infty(\partial\varOmega)\rightarrow L^\infty(\varOmega)}
&\le
C\tau \int_{\Gamma_{\theta,\sigma}\backslash\Gamma_{\theta,\sigma}^\tau} 
e^{-\tau {\rm Re}(z)} \| \tilde L_h(e^{-\tau z}) \|_{L^\infty(\partial\varOmega)\rightarrow L^\infty(\varOmega)} e^{t_n{\rm Re}(z)} |\d z | \\
&\le
C\tau \int_{\Gamma_{\theta,\sigma}\backslash\Gamma_{\theta,\sigma}^\tau} 
e^{-t_{n-1}|z||\cos(\theta)|} |\d z | \\
&\le 
C\tau \int_{\frac{\pi}{\tau\sin(\theta)}}^{+\infty} 
e^{-t_{n-1} r |\cos(\theta)|} \d r \\ 
&\le 
C\int_{\frac{\pi}{\sin(\theta)}}^{+\infty} 
e^{-\frac{t_{n-1} |\cos(\theta)|}{\tau} \tilde r } \d \tilde r \quad\mbox{(change of variables)} \\ 
&\le 
C\frac{1}{(n-1) |\cos\theta|}
e^{-(n-1)\pi/|\tan(\theta)|} \quad\mbox{for}\,\,\, n\ge 2 , 
\end{align*}
and then 
\begin{align}\label{Mh-ghj}
\bigg\| \sum_{j=k}^{N-2} L^{N-j}_h g_h^j \bigg\|_{L^\infty(\varOmega)}
&\le
C\sum_{j=k}^{N-2} (N-1-j)^{-1} e^{-(N-1-j)\pi/|\tan(\theta)|} 
\max_{k\le n\le N-2}\| g_h^n \|_{L^\infty(\partial\varOmega)}  
\notag \\
&\le
C\max_{k\le n\le N-2}\| g_h^n \|_{L^\infty(\partial\varOmega)} .
\end{align}

Since $\tau|z|\ge \pi/\sin(\theta)$ for $z\in
\Gamma_{\theta,\sigma}\backslash\Gamma_{\theta,\sigma}^\tau$, by using the second inequality of \eqref{exp-tau-z} it is easy to verify that 
\begin{align*} 
\bigg|\frac{1-e^{-\tau z}}{\tau}-z\bigg|
\le
C\bigg|\frac{1-e^{-\tau z}}{\tau}\bigg|
\le
C\tau |z| \bigg|\frac{1-e^{-\tau z}}{\tau}\bigg|
\quad\forall\,z\in
\Gamma_{\theta,\sigma}\backslash\Gamma_{\theta,\sigma}^\tau .
\end{align*}
With this inequality and taking $\sigma = s$, we have
\begin{align*} 
	\| K_h(s) \|_{L^\infty(\varOmega)\rightarrow L^\infty(\varOmega)}
	&\le 
	C \int_{\Gamma_{\theta,\sigma}^\tau} 
	\bigg|\frac{1-e^{-\tau z}}{\tau}\bigg|^{-1} \tau |z|^2 e^{s{\rm Re}(z)} |\d z| \notag\\
	&\quad
	+ C \int_{\Gamma_{\theta,\sigma}\backslash \Gamma_{\theta,\sigma}^\tau} 
	\bigg|\frac{1-e^{-\tau z}}{\tau}\bigg|^{-1} \tau |z|
	\bigg|\frac{1-e^{-\tau z}}{\tau}\bigg| e^{s{\rm Re}(z)} |\d z| \notag\\
	&\le 
	C\tau \int_{\Gamma_{\theta,s}} |z| e^{s{\rm Re}(z)} |\d z| \notag\\
	&\le 
	C\tau \int_{s^{-1}}^{+\infty} r e^{-s r|\cos(\theta)|} \d r 
	+C\tau \int_{-\theta}^{\theta} s^{-2} e^{-|\cos\theta|} \d \varphi \notag\\
	&\le C\tau s^{-2} \quad\mbox{for}\,\,\, s\ge \tau . 
\end{align*}
Hence, 
\begin{align}\label{Kh-uhs}
\bigg\| \int_{t_{k-1}}^{t_{N-1}}K_{h}(t_N-s) P_h u_{h,2}(s)\d s \bigg\|_{L^\infty(\varOmega)}
&\le
\int_{t_{k-1}}^{t_{N-1}} C\tau (t_N-s)^{-2} \| u_{h,2}(s) \|_{L^\infty(\varOmega)} \d s
\notag \\[2pt]
&\le
C \| u_{h,2} \|_{L^\infty(t_{k-1},t_{N-1};L^\infty(\varOmega))} \notag \\[2pt]
&\le
C\| g_h \|_{L^\infty(t_{k-1},t_{N-1};L^\infty(\partial\varOmega))} \notag \\
&=
C\max_{k\le n\le N-2} \| g_h^n \|_{L^\infty(\partial\varOmega)}  ,
\end{align}
where we have used Theorem \ref{thm:semi-WMP} in the last inequality. 

Similarly, by choosing $\sigma=\tau$ in the expression of $Q_h$, we have 
\begin{align*} 
\| Q_h \|_{L^\infty(\varOmega)\rightarrow L^\infty(\varOmega)}
&\le 
C \int_{\Gamma_{\theta,\sigma}^\tau} 
\bigg|\frac{1-e^{-\tau z}}{\tau}\bigg|^{-1} \tau |z|^2 |z|^{-1} e^{\tau {\rm Re}(z)} |\d z| \notag\\
&\quad
+ C \int_{\Gamma_{\theta,\sigma}\backslash \Gamma_{\theta,\sigma}^\tau} 
\bigg|\frac{1-e^{-\tau z}}{\tau}\bigg|^{-1} \tau |z|
\bigg|\frac{1-e^{-\tau z}}{\tau}\bigg| |z|^{-1} e^{\tau {\rm Re}(z)} |\d z| \notag\\
&\le 
C\tau \int_{\Gamma_{\theta,\tau}} e^{\tau{\rm Re}(z)} |\d z| \notag\\
&\le 
C\tau \int_{\tau^{-1}}^{+\infty} e^{-\tau r|\cos(\theta)|} \d r 
+C\tau \int_{-\theta}^{\theta} \tau^{-1} e^{-|\cos\theta|} \d \varphi \notag\\
&\le C. 
\end{align*}
Hence,
\begin{align}\label{Qh-uhs}
\|Q_{h}u_{h,2}(t_{N-1})\|_{L^\infty(\varOmega)}
\le
C\|u_{h,2}(t_{N-1})\|_{L^\infty(\varOmega)} 
&\le
C\| g_h \|_{L^\infty(0,t_{N-1};L^\infty(\partial\varOmega))} \notag \\
&=
C\max_{k\le n\le N-2} \| g_h^n \|_{L^\infty(\partial\varOmega)} , 
\end{align}
where we have used Theorem \ref{thm:semi-WMP} again in the last inequality. 

Finally, substituting \eqref{Fh-ghj}--\eqref{Qh-uhs} into \eqref{BDF-FEM-difference}, we obtain 
\begin{align}\label{BDF-FEM-diff-f1}
\| u_{h,2}^N - u_{h,2}(t_N) \|_{L^\infty(\varOmega)}
\le
C\max_{k\le n\le N-2} \| g_h^n \|_{L^\infty(\partial\varOmega)} . 
\end{align}
Then, using the triangle inequality and Theorem \ref{thm:semi-WMP}, we obtain \eqref{BDF-FEM-f}.

\subsection{Proof of \eqref{Eq-uh3-Linfty}}
For $k \leq N \leq 2k - 1$,  at the time $t_k$, we solve for $u_h^k\in \mathring S_h$ satisfying
\begin{align}
	\bigg( \frac{u_h^{k}}{\tau}, v_h\bigg) + (\nabla u_h^n,\nabla v_h) = -\bigg( \frac{1}{\tau} \sum_{j=1}^k \delta_j u_h^{n-j} ,v_h\bigg) \forall\, v_h\in \mathring S_h , \notag
\end{align}
which can be equivalently written as
\begin{align}
	u_h^{k} = -\bigg( \frac{1}{\tau} -  \Delta_h \bigg)^{-1} P_h \frac{1}{\tau} \sum_{j=1}^k \delta_j u_h^{k-j} . \notag
\end{align}
Then,
\begin{align}
	\| u_h^{k} \|_{L^\infty(\Omega)} 
	&\leq  \bigg\| \bigg( \frac{1}{\tau} -  \Delta_h \bigg)^{-1}  \bigg\|_{L^\infty(\Omega) \rightarrow L^\infty(\Omega)} \bigg\| P_h \frac{1}{\tau} \sum_{j=1}^k \delta_j u_h^{k-j}  \bigg\|_{L^\infty(\Omega)} \notag\\
	&\leq C \sup_{0\leq j \leq k-1} \| u_{h}^j \|_{L^\infty(\Omega)} . \notag
\end{align}
Analogously, at the time $t_{k+1}$, we have
\begin{align}
	u_h^{k+1} = -\bigg( \frac{1}{\tau} -  \Delta_h \bigg)^{-1} P_h \frac{1}{\tau} \sum_{j=1}^k \delta_j u_h^{k+1-j} , \notag
\end{align}
with
\begin{align}
	\| u_h^{k+1} \|_{L^\infty(\Omega)} 
	\leq C \sup_{1\leq j \leq k} \| u_{h}^j \|_{L^\infty(\Omega)} \leq C \sup_{0\leq j \leq k-1} \| u_{h}^j \|_{L^\infty(\Omega)} . \notag
\end{align}
Recursively, we conclude
\begin{align}
	\| u_h^{N} \|_{L^\infty(\Omega)} 
	\leq C \sup_{0\leq j \leq k-1} \| u_{h}^j \|_{L^\infty(\Omega)} , \notag
\end{align}
for any $k\leq N\leq 2k-1$.

If $N \geq 2k$, multiplying \eqref{Eq-uh3} by $\zeta^n$, summing up the results for $n=k,k+1\dots,$ and using the identity \eqref{eq:gen_iden}, we obtain the following elliptic equation for the generating function $\tilde u_{h,3}(\zeta)\in \mathring S_h(\Omega)$ satisfying zero Dirichlet boundary condition: 
$$
\bigg( \delta_\tau(\zeta) \tilde u_{h,3}(\zeta),v_h\bigg) + (\nabla \tilde u_{h,3}(\zeta),\nabla v_h)= \bigg( - \frac{1}{\tau} \sum_{n=k}^{2k-1} \zeta^n \sum_{j=0, n-j < k}^k \delta_j u_{h,3}^{n-j},v_h\bigg) ,
$$
whose solution can be represented by
$$
\tilde u_{h,3}(\zeta) 
= - \bigg(\delta_\tau(\zeta) - \varDelta_h\bigg)^{-1} 
 P_h \frac{1}{\tau} \sum_{n=k}^{2k-1} \zeta^n \sum_{j=0, n-j < k}^k \delta_j u_{h,3}^{n-j} .
$$
$\tilde u_{h, 3}(\zeta)$ is analytic in $\D = \{ \zeta\in\C : |\zeta|< 1 \}$ due to the standard energy estimate. By Cauchy's integral formula and the deformation of contour (cf. \eqref{u_{h,2}n-repr}), for any $0<\rho<1$ we have 
\begin{align*}
u_{h,3}^N 
&= 
- \frac{1}{2\pi i}
\int_{|\zeta|=\rho}
\bigg(\delta_\tau(\zeta) - \varDelta_h\bigg)^{-1} 
P_h \frac{1}{\tau} \sum_{n=k}^{2k-1} \zeta^n \sum_{j=0, n-j < k}^k \delta_j u_{h,3}^{n-j} \zeta^{-N-1} \d \zeta \notag\\ 
&= 
- \frac{1}{2\pi i}
\int_{\Sigma_{\theta,\sigma}^\tau}
\bigg(\delta_\tau(e^{-\tau z}) - \varDelta_h\bigg)^{-1} 
P_h \sum_{n=k}^{2k-1} \sum_{j=0, n-j < k}^k \delta_j u_{h,3}^{n-j} e^{t_{N-n} z} \d z .
\end{align*}
By choosing $\sigma=t_{N-n}$ ($N - n \geq 2k - (2k - 1) = 1$) for each term of the summation in the representation formula above, it follows 
\begin{align}
	&\| u_{h,3}^N \|_{L^\infty(\Omega)} \notag\\
	&\leq C \int_{\Gamma_{\theta,\sigma}^\tau}
	\| (\delta_\tau(\zeta) - \Delta_h)^{-1} \|_{L^\infty(\Omega) \rightarrow L^\infty(\Omega)} \sum_{n=k}^{2k-1} \sum_{j=0, n-j < k}^k \| P_h  \delta_j u_{h,3}^{n-j} \|_{L^\infty(\Omega)} e^{t_{N-n} {\rm Re} (z)}\d z \notag\\
	&\leq C \sup_{0\leq j \leq k-1} \| u_{h}^j \|_{L^\infty(\Omega)} \sum_{n=k}^{2k-1} \bigg(\int_{\Gamma_{\theta,t_{N-n}}^{\tau,1}} + \int_{\Gamma_{\theta,t_{N-n}}^{\tau, 2}} \bigg)
	|z|^{-1} e^{t_{N-n}{\rm Re} (z)}\d z \notag\\
	&\leq C \sup_{0\leq j \leq k-1} \| u_{h}^j \|_{L^\infty(\Omega)} \sum_{n=k}^{2k-1} \bigg(\int_{t_{N-n}^{-1}}^{+\infty} 
	r^{-1} e^{-t_{N-n} r|\cos\theta|} \d r + \int_{-\theta}^\theta
	e^{-|\cos\theta|} \d\theta \bigg) \notag\\
	&\leq C \sup_{0\leq j \leq k-1} \| u_{h}^j \|_{L^\infty(\Omega)} , \notag
\end{align}
where in the second inequality we have applied the resolvent estimate (Lemma \ref{lemma:BDF_sec}) and the $L^\infty$ stability of $P_h$.

This proves \eqref{Eq-uh3-Linfty}. The proof of Theorem \ref{THM:BDF-FEM} is complete.

\section{Conclusions and further discussions}
\setcounter{equation}{0}

We have proved the weak maximum principle of semi-discrete FEM and fully discrete FEM with $k$-step BDF time-stepping method for $k=1,\dots,6$, using the techniques of generating polynomials and the discrete inverse Laplace transform (Cauchy integral formula), which requires using a uniform time step size. Therefore, the analysis in this article cannot be readily extended to $k$-step BDF methods or other single-step methods (such as the Runge--Kutta methods or discontinous Galerkin (dG) methods) with variable time step sizes $\tau_n$, $n=1,\dots,N$. For the dG time-stepping method with possibly variable step sizes satisfying the following conditions:
\begin{enumerate}
  \item There are constants $c,\beta>0$ independent of the maximal step size $\tau$ such that $\tau_{\min}\ge c\tau^\beta$. 
  
  \item There is a constant $\kappa>0$ independent on $\tau$ such that 
    $
    \kappa^{-1}\le\frac{\tau_n}{\tau_{n+1}}\le \kappa 
    $ for $n=1,\dots,N-1$. 
  \item It holds $\tau\le\frac{1}{4}T$, 
\end{enumerate} 
an application of the discrete semigroup estimates in \cite{Ley17} can yield the following result:
$$
\| u_{kh}\|_{L^\infty(I;L^\infty(\Omega))}\le C\ln{\frac{T}{\tau}}\, \| g_h\|_{L^\infty(I;L^\infty(\partial \Omega))}, 
$$
with an additional logarithmic factor depending on the maximal step size $\tau$. We present a proof for dG(0) in Appendix. The general dG(k) can be treated similarly by using the approach from \cite{Ley17}. 
The proof of weak maximum principle of dG with variable step sizes without the logarithmic factor is interesting and nontrivial.

\section*{Acknowledgment} 
We thank the anonymous referee for the valuable comments. The work of G. Bai and B. Li is supported in part by the Research Grants Council of the Hong Kong Special Administrative Region, China (GRF project no. PolyU15300519) and an internal grant at The Hong Kong Polytechnic University (Project ID: P0045404).  

\appendix
\section*{\large Appendix: Weak maximum principle of dG with variable step sizes}
\renewcommand{\theequation}{A.\arabic{equation}}
\renewcommand{\thetheorem}{A}
\renewcommand{\theproposition}{A}
\renewcommand{\thefigure}{A}
\setcounter{equation}{0}
\setcounter{theorem}{0}
\normalsize

The space-time finite element space of degree $0$ in time and degree $r\ge 1$ in space is defined as 
\begin{equation} \label{def: space_time}
X_{\tau,h}^{0,r} =\{v_{\tau h} :\ v_{\tau h}|_{I_n}\in \Ppol{0}(\mathring S_h), \ n=1,2,\dots,N, \,\,\,  r\geq 1 \}.
\end{equation}
The dG($0$) method for the homogeneous parabolic equation 
\begin{align}\label{PDE HOMO}
\left\{
\begin{aligned}
&\partial_t v - \Delta v=0  &&\mbox{in}\,\,\,(0,T] \times \Omega,\\ 
&v=0 &&\mbox{on}\,\,\,[0,T] \times \partial\Omega ,\\
&v|_{t=0}=v_0 &&\mbox{in}\,\,\,\Omega .
\end{aligned}
\right.	
\end{align}
can be written as follows: For the given $v_0\in L^p(\Omega)$, find $v_{\tau h}\in X_{\tau,h}^{0,r}$ such that 
\begin{equation}\label{eq: homogeneous dg0 parabolic one step fully}
\begin{aligned}
v_{\tau h,1}-\tau_1\Delta_h v_{\tau h,1}&=v_0,\\
v_{\tau h,m}-\tau_m\Delta_h v_{\tau h,m}&=v_{\tau h,m-1}, \quad m=2,3,\dots,M.
\end{aligned}
\end{equation}
The following result was shown in \cite{Ley17}.
\begin{proposition}[Discrete semigroup estimates in $L^p(\Omega)$]\label{prop: homogenous laplaca}
Let $v_{\tau h}\in X_{\tau,h}^{0,r}$ be the solution of \eqref{eq: homogeneous dg0 parabolic one step fully} with $v_0 \in L^p(\Omega)$, $ 1\le p\le \infty$. Then there exists a constant $C$ independent of $\tau$ such that
$$
\|v_{\tau h,m}\|_{L^p(\Omega)} + (t_m-t_l)\|\Delta_h v_{\tau h,m}\|_{L^p(\Omega)}\le C\|v_{\tau h,l}\|_{L^p(\Omega)},\quad \forall m>l\ge 1.
$$
\end{proposition}
Now we consider $u_{\tau h}\in X_{\tau,h}^{0,r}$ to be the dG($0$) solution to the parabolic equation with nonzero right hand side $f$, i.e., $u_{\tau h}$ satisfies,
\begin{equation}\label{eq: one step dG0 inhomogeneous}
\begin{aligned}
u_{kh,1}-\tau_1\Delta_h u_ {kh,1} &= \tau_1P_hf_1,\\
u_{kh,m}-\tau_m\Delta_h u_ {kh,m} &= u_{kh,m-1}+\tau_mP_h f_m,\quad m=2,3,\dots,M,
\end{aligned}
\end{equation}
where
$$
f_m(\cdot)=\frac{1}{\tau_m}\int_{I_m}f(t,\cdot)dt
$$
and $P_h: L^2(\Omega)\to \mathring S_h$ is the $L^2$ projection. 
Since $f_m$  is  the $L^2$ projection  of $f$ onto the piecewise constant functions on each subinterval $I_m$, we have
\begin{equation} \label{eq: estimate for fm in lp}
\max_{1\le m\le M}\|f_m\|_{L^p(\Omega)}\le C\|f\|_{L^\infty(I;L^p(\Omega))}, \quad 1\le p\le\infty.
\end{equation}
Using \eqref{eq: one step dG0 inhomogeneous}, we can write the dG($0$) solution as
\begin{equation}\label{dg0 inhomogeneous}
u_{kh,m}=\sum_{l=1}^m \tau_l\left(\prod_{j=1}^{m-l+1}r(-\tau_{m-j+1}\Delta_h)\right)P_hf_l,\quad m=1,2,\dots,M,
\end{equation}
where $r(z)=(1+z)^{-1}.$

Consider the parabolic equation 
\begin{align}\label{PDE}
\left\{
\begin{aligned}
&\partial_t u - \Delta u=0 &&\mbox{in}\,\,\,(0,T] \times \Omega,\\ 
&u=g &&\mbox{on}\,\,\,[0,T] \times \partial\Omega ,\\
&u|_{t=0}=0 &&\mbox{in}\,\,\,\Omega .
\end{aligned}
\right.	
\end{align}
With the help of the previous result we establish:
\begin{theorem}[Weak maximum principle of dG(0)]\label{thm: maximal parabolic}
Let $u_0=0$. Then, there exists a constant $C$ independent of $k$ and $h$ such that for 
  $u_{kh}$ be the fully discrete dG(0) solution to \eqref{PDE}, we have
$$
\| u_{kh}\|_{L^\infty(I;L^\infty(\Omega))}\le C\ln{\frac{T}{\tau}}\| g_h\|_{L^\infty(I;L^\infty(\partial \Omega))},
$$
where $g_h$ is some finite element approximation of $g$. 
\end{theorem}
\begin{proof}
First we define Let $G_h$ be the discrete harmonic extension of $g_h$, i.e. $G_h(t)$ satisfies (pointwise in $t$)
\begin{align}\label{PDE harmonic}
\left\{
\begin{aligned}
&(\nabla G_h,\nabla \chi)_\Omega =0 &&\forall \chi\in \mathring S_h,\\ 
&G_h=g_h &&\mbox{on}\,\,\, \partial\Omega.
\end{aligned}
\right.	
\end{align}
Then  the difference $v_h=u_h-G_h\in \mathring S_h$ (zero on $\partial \Omega$) satisfies the semi-discrete problem, 
\begin{align}\label{PDE semi-discrete new}
\left\{
\begin{aligned}
&(\partial_t v_h,\chi)_\Omega +(\nabla v_h,\nabla \chi)_\Omega =(\partial_t G_h,\chi)_\Omega &&\mbox{in}\,\,\,(0,T] \times \Omega,\\ 
&v_h|_{t=0}=0 &&\mbox{in}\,\,\,\Omega ,
\end{aligned}
\right.	
\end{align}
for any $\chi\in \mathring S_h$, and as a result the fully discrete difference
$v_{kh}=u_{kh}-P_hG_{kh}\in X_{k,h}^{0,r}$, where $P_h$ is the $L^2$-orthogonal projection onto $\mathring S_h$, using \eqref{eq: one step dG0 inhomogeneous}  satisfies 
$$
v_{kh,m}=\sum_{l=1}^m \tau_l\left(\prod_{j=1}^{m-l+1} r(-\tau_{m-j+1}\Delta_h)\right)\frac{P_hG_{kh,l}-P_hG_{kh,l-1}}{\tau_l},\quad m=1,2,\dots,M,
$$
where $r(z)=(1+z)^{-1}$ and with the convention that $G_{kh,0}=0$.
Using the discrete integration by parts, we have 
$$
\begin{aligned}
&\sum_{l=1}^m \left(\prod_{j=1}^{m-l+1} r(-\tau_{m-j+1}\Delta_h)\right)\left(P_hG_{kh,l}-P_hG_{kh,l-1}\right)\\
=&r(-\tau_{m}\Delta_h)P_hG_{kh,m}-\prod_{j=1}^{m}r(-\tau_{m-j+1}\Delta_h)P_hG_{kh,0}\\
&+\sum_{l=2}^{m-1} (\operatorname{Id}-r(-\tau_{l}\Delta_h))\left(\prod_{j=1}^{m-l+1}r(-\tau_{m-j+1}\Delta_h)\right)P_hG_{kh,l}\\
=&r(-\tau_{m}\Delta_h)P_hG_{kh,m}+\sum_{l=2}^{m-1} \tau_{l}\Delta_hr(-\tau_{l}\Delta_h)\left(\prod_{j=1}^{m-l+1}r(-\tau_{m-j+1}\Delta_h)\right)P_hG_{kh,l},
\end{aligned}
$$
where we used that $ \operatorname{Id}-r(-\tau_{l}\Delta_h)=\tau_{l}\Delta_hr(-\tau_{l}\Delta_h)$ and $P_hG_{kh,0}=0$.
Using the properties of the rational function $r$ and the elliptic theory, we have
$$
\|r(-\tau_{m}\Delta_h)G_{kh,m+1}\|_{L^\infty(\Omega)}
\le C  \|g\|_{L^\infty(I;L^\infty(\partial\Omega))}.
$$
Hence for $ m=1,2,\dots,M$,
$$
\begin{aligned}
\| v_{kh,m}\|_{L^\infty(\Omega)}\le C  \|g\|_{L^\infty(I;L^\infty(\partial\Omega))}+\sum_{l=1}^m \tau_l\left\|\left(\Delta_h\prod_{j=1}^{m-l+1} r(-\tau_{m-j+1}\Delta_h)\right)P_hG_{kh,l}\right\|_{L^\infty(\Omega)}.
\end{aligned}
$$
From Proposition \ref{prop: homogenous laplaca}, since each term in the sum on the right-hand side can be thought of as a homogeneous solution with initial condition $P_hG_{kh,l}$ at $t=t_{l-1}$, we have
$$
\left\|\left(\Delta_h\prod_{j=1}^{m-l+1} r(-\tau_{m-j+1}\Delta_h)\right)P_hG_{kh,l}\right\|_{L^{\infty}(\Omega)}\le \frac{C}{t_m-t_{l-1}}\|G_{kh,l}\|_{L^{\infty}(\Omega)}.
$$
Thus, we obtain
\begin{equation}\label{eq: before Holder}
\| v_{kh,m}\|_{L^{\infty}(\Omega)}\le C  \|g\|_{L^\infty(I;L^\infty(\partial\Omega))}+ C\sum_{l=1}^m  \frac{\tau_l}{t_m-t_{l-1}}\|G_{kh,l}\|_{L^{\infty}(\Omega)},\quad m=1,2,\dots,M.
\end{equation}
From the above estimate and using \eqref{eq: estimate for fm in lp},
\begin{equation*}
\begin{aligned}
\| v_{kh}\|_{L^\infty(I;L^\infty(\Omega))}&=\max_{1\le m\le M}\| v_{kh,m}\|_{L^\infty(\Omega)}\\
&\le  C  \|g\|_{L^\infty(I;L^\infty(\partial\Omega))}+C\max_{1\le m\le M}\sum_{l=1}^m \frac{\tau_l}{t_m-t_{l-1}}\|G_{kh,l}\|_{L^\infty(\Omega)}\\
&\le  C  \|g\|_{L^\infty(I;L^\infty(\partial\Omega))}+C\max_{1\le l\le M}\|G_{kh,l}\|_{L^\infty(\Omega)}\max_{1\le m\le M}\sum_{l=1}^m  \frac{\tau_l}{t_m-t_{l-1}}\\
&\le C\ln{\frac{T}{\tau}} \|g\|_{L^\infty(I;L^\infty(\partial\Omega))},
\end{aligned}
\end{equation*}
where in the last step we used that 
$$
\max_{1\le l\le M}\|G_{kh,l}\|_{L^\infty(\Omega)}\le C\|g\|_{L^\infty(I;L^\infty(\partial\Omega))}
$$
 and
\begin{equation}\label{eq: estimating sum by integral for log}
\sum_{l=1}^m  \frac{\tau_l}{t_m-t_{l-1}}\le 1+\int_0^{t_{m-1}}\frac{dt}{t_m-t}=1+\ln{\frac{t_m}{\tau_m}}\le C\ln{\frac{T}{\tau}},
\end{equation}
by using the assumption $\tau_{\min}\geq C \tau^\beta$ and $\tau \le \frac{T}{4}$.
Finally, using the triangle inequality, it follows that
$$
\| u_{kh}\|_{L^\infty(I;L^\infty(\Omega))}\le \| v_{kh}\|_{L^\infty(I;L^\infty(\Omega))}+\| P_hG_{kh}\|_{L^\infty(I;L^\infty(\Omega))}\le C\ln{\frac{T}{\tau}} \|g\|_{L^\infty(I;L^\infty(\partial\Omega))}.
$$
Thus we obtain the result. 
\end{proof}
\medskip


\end{document}